\documentclass[12pt]{smfart}
\usepackage{color}
\usepackage{amssymb,verbatim}
\usepackage{amsthm,array,amssymb,amscd,amsfonts,amssymb,latexsym, url}
\usepackage{amsmath}
\usepackage[all]{xy}
\usepackage[french]{babel}
\usepackage{times}

\newcounter{spec}
\newenvironment{thlist}{\begin{list}{\rm{(\roman{spec})}}%
{\usecounter{spec}\labelwidth=20pt\itemindent=0pt\labelsep=10pt}}%
{\end{list}}

\newcommand{\surj}{\rightarrow\!\!\!\!\!\rightarrow}
\newcommand{\Surj}{\relbar\joinrel\surj}

\renewcommand{\P}{{\mathbf P}}

\newcommand{\sA}{{\mathcal A}}

\newcommand{\sM}{{\mathcal M}}

\newcommand{\sO}{{\mathcal O}}

\newcommand{\bH}{\mathbb{H}}

\newcommand{\car}{\operatorname{car}}

\newcommand{\R}{{\mathbb R}}
\newcommand{\Z}{{\mathbb Z}}
\newcommand{\Q}{{\mathbb Q}}
\newcommand{\C}{{\mathbb C}}
\renewcommand{\H}{{\mathcal H}}
\renewcommand{\L}{{\mathbb L}}

\newcommand{\ovF}{{\overline {\F}}}

\newcommand{\ovV}{{\overline V}}
\newcommand{\ovX}{{\overline X}}

\newcommand{\K}{{\mathcal K}}

\newcommand{\cl}{{\operatorname{cl}}}
\newcommand{\Br}{{\operatorname{Br}}}
\newcommand{\Ker}{{\operatorname{Ker}}}
\newcommand{\Coker}{{\operatorname{Coker}}}
\newcommand{\Hom}{{\operatorname{Hom}}}
\newcommand{\Spec}{{\operatorname{Spec \ }}}
\newcommand{\prodr}{\operatornamewithlimits{\prod\!\!\!\!\!\!\!\!\coprod}}

\newcommand{\num}{{\operatorname{num}}}
\newcommand{\nr}{{\operatorname{nr}}}
\renewcommand{\div}{{\operatorname{div}}}
\newcommand{\tors}{{\operatorname{tors}}}

\newcommand{\cd}{{\operatorname{cd}}}
\newcommand{\eff}{{\operatorname{eff}}}

\newcommand{\inv}{{\operatorname{inv}}}
\renewcommand{\o}{{\operatorname{o}}}
\newcommand{\Tate}{{\operatorname{Tate}}}

\newcommand{\by}[1]{\overset{#1}{\longrightarrow}}
\newcommand{\iso}{\by{\sim}}
\renewcommand{\lim}{\varprojlim}
\newcommand{\colim}{\varinjlim}

\renewcommand{\phi}{\varphi}

\numberwithin{equation}{section}
\newcommand{\cqfd}
{%
\mbox{}%
\nolinebreak%
\hfill%
\rule{2mm}{2mm}%
\medbreak%
\par%
}
\newfont{\gothic}{eufb10}

\renewcommand{\qed}{{\hfill$\square$}}

\newtheorem{theo}{Th\'{e}or\`{e}me}[section]
\newtheorem{prop}[theo]{Proposition}
\newtheorem{conj}[theo]{Conjecture}
\newtheorem{lem}[theo]{Lemme}
\newtheorem{cor}[theo]{Corollaire}
\theoremstyle{definition}
\newtheorem{defi}[theo]{D\'efinition}
\theoremstyle{remark}
\newtheorem{rema}[theo]{Remarque}

\newtheorem{qn}[theo]{Question}
\newtheorem{propr}[theo]{Propri\'et\'es}

\newcommand{\bthe}{\begin{theo}}
\newcommand{\ble}{\begin{lem}}
\newcommand{\bpr}{\begin{prop}}
\newcommand{\bco}{\begin{cor}}
\newcommand{\bde}{\begin{defi}}
\newcommand{\ethe}{\end{theo}}
\newcommand{\ele}{\end{lem}}
\newcommand{\epr}{\end{prop}}
\newcommand{\eco}{\end{cor}}
\newcommand{\ede}{\end{defi}}

\newcommand{\Gal}{{\rm{Gal}}}
\newcommand{\et}{   {\operatorname{\acute{e}t}   }   }
\newcommand{\Zar}{{\operatorname{Zar}}}
\newcommand{\Pic}{\operatorname{Pic}}
\newcommand{\F}{{\mathbb F}}

\DeclareFontFamily{U}{wncy}{}
\DeclareFontShape{U}{wncy}{m}{n}{%
<5>wncyr5%
<6>wncyr6%
<7>wncyr7%
<8>wncyr8%
<9>wncyr9%
<10>wncyr10%
<11>wncyr10%
<12>wncyr6%
<14>wncyr7%
<17>wncyr8%
<20>wncyr10%
<25>wncyr10}{}
\DeclareMathAlphabet{\cyr}{U}{wncy}{m}{n}
\begin{document}

\title[Cycles de codimension $2$ et $H^3$ non ramifi\'e]{Cycles de codimension 2 et $H^3$ non ramifi\'e pour les vari\'et\'es sur les corps finis}
\author{Jean-Louis Colliot-Th\'el\`ene}
\address{C.N.R.S., Universit\'e Paris Sud\\Math\'ematiques, B\^atiment 425\\91405 Orsay Cedex\\France}
\email{jlct@math.u-psud.fr}
\author{Bruno Kahn}
\address{Institut de Math\'ematiques de Jussieu\\UMR 7586\\ Case 247\\4 place
Jussieu\\75252 Paris Cedex 05\\France}
\email{kahn@math.jussieu.fr}
\date{Soumis au Journal of K-Theory, 28  janvier 2012; version r\'evis\'ee  soumise  le 7 juin 2012 et  accept\'ee le 23 juillet 2012}
\maketitle

 \begin{abstract}
  Soit $X$ une vari\'et\'e projective et lisse sur un corps fini $\F$. On s'int\'eresse
au groupe de cohomologie non ramifi\'ee $H^3_\nr(X,\Q/\Z(2))$.
Un faisceau de conjectures implique que ce groupe  
est fini.  Pour certaines classes de solides, on a $H^3_\nr(X,\Q/\Z(2))=0$. 
  Savoir si c'est le cas pour tout solide (vari\'et\'e de dimension 3) est un probl\`eme ouvert.
 Lorsqu'un solide $X$  est fibr\'e au-dessus d'une courbe $C$, de fibre g\'en\'erique une surface projective et lisse $V$
sur le corps global $\F(C)$, la combinaison de $H^3_\nr(X,\Q/\Z(2))=0$ 
et de la conjecture de Tate pour les diviseurs sur $X$ a pour cons\'equence  un principe local-global 
de type Brauer--Manin pour le groupe de Chow des z\'ero-cycles de la fibre g\'en\'erique $V$. 
 Ceci \'eclaire d'un jour nouveau 
des   investigations commenc\'ees
il y a  trente ans.

\end{abstract}

\begin{altabstract}
Let $X$ be a smooth projective variety over a finite field $\F$. We discuss the
unramified cohomology group $H^3_\nr(X,\Q/\Z(2))$. Several conjectures 
put together imply
that this group is finite. For certain classes of threefolds, $H^3_\nr(X,\Q/\Z(2))$  actually vanishes. 
It is an open question whether this holds   for arbitrary threefolds. 
For a threefold $X$ equipped with a fibration onto a curve $C$, 
the generic fibre of which is a smooth projective surface $V$ over the global
field $\F(C)$, the vanishing of  $H^3_\nr(X,\Q/\Z(2))$  together with the Tate conjecture
for divisors on $X$ implies a local-global principle of Brauer--Manin type
for the Chow group of zero-cycles on $V$. 
This sheds a new light on 
work started thirty years ago.
\end{altabstract}
 
  \tableofcontents
 
 \section*{Introduction}

\`A une  vari\'et\'e $X$ projective, lisse, et g\'eom\'etriquement connexe sur un corps $k$, on associe (\cite{ct-o}, \cite{ctbarbara}) des
 groupes de cohomologie non ramifi\'ee $$H^{i}_{\nr}(X,\Q/\Z(i-1)), \hskip2mm i \geq 1,$$ qui sont  
des invariants birationnels. 
 Pour $i=1$ ce groupe   classifie les rev\^etements ab\'eliens \'etales de $X$.
Pour $i=2$, on obtient le groupe de Brauer $\Br(X)$. 
 Pour tout $i$, ils satisfont une propri\'et\'e de rigidit\'e
\cite[Thm. 4.4.1]{ctbarbara}.

On s'int\'eresse ici au groupe $H^{3}_{\nr}(X,\Q/\Z(2))$ et \`a ses liens  avec le groupe
de Chow $CH^2(X)$ des cycles de codimension deux modulo l'\'equiva\-lence rationnelle.
Ceci a d\'ej\`a fait  l'objet de deux articles r\'ecents, l'un de Claire Voisin et du premier
auteur \cite{ctv}, l'autre du deuxi\`eme auteur \cite{kahncycle}. 
 Dans ces articles,
il est \'etabli que :

\begin{itemize}

\item  Pour $k$ un corps fini ou alg\'ebriquement clos,  
pour tout nombre premier $l$  diff\'erent de la caract\'eristique, le grou\-pe
$H^{3}_{\nr}(X,\Q_{l}/\Z_{l}(2))$ est extension d'un groupe fini $C_{l}$ par 
un groupe divisible. 
 Le  groupe  $C_{l}$ est le sous-groupe de torsion du conoyau de l'application cycle  
 $ CH^{2}(X) \otimes \Z_{l} \to H^4_{\et}(X,\Z_{l}(2)).$

\item Si $k=\F$ est un corps fini,   $H^{3}_{\nr}(X,\Q/\Z(2))$ est fini si $X$ appartient \`a
une certaine classe $B_\Tate(\F)$  qui contient conjecturalement toutes les vari\'et\'es
projectives et lisses.

\end{itemize}

\medskip

Il se trouve, et c'est le point de d\'epart de cet article, que le groupe $H^{3}_{\nr}(X,\Q/\Z(2))$
 intervient dans l'\'etude des z\'ero-cycles sur les vari\'et\'es d\'efinies sur un corps  global. 
Dans ce domaine, on a la conjecture suivante  \cite{cts81,  katosaito, saito, ctseattle}:

\medskip

\noindent {\bf \textit{Conjecture \ref{congenzerocycle}.}}
 {\it
Soit $K$ un corps global. Soit $l$ un nombre premier diff\'erent de la caract\'eristique de $K$. Soit $V$ une vari\'et\'e projective et lisse,
g\'eom\'etriquement connexe sur $K$. S'il existe une famille $\{z_{v}\}$ de z\'ero-cycles locaux
 de
degr\'es premiers \`a $l$  telle que, pour tout \'el\'ement  
$A \in \Br(V)\{l\}$, on ait $$\sum_{v} \inv_{v}(A(z_{v}))=0,$$ alors il existe un z\'ero-cycle de
degr\'e premier \`a $l$  sur $V$.}

Cette conjecture a fait l'objet d'un certain nombre de travaux, l'un des plus r\'ecents \'etant l'article    \cite{wittenberg}  de O.~Wittenberg.

\medskip

Un lien entre la conjecture \ref{congenzerocycle} et la cohomologie non ramifi\'ee de degr\'e $3$ est donn\'e par l'\'enonc\'e suivant.

\medskip

\noindent {\bf \textit{Th\'eor\`eme}} (\emph{cf}. Th\'eor\`eme  \ref{corollairebrauermaninF(C)}).  
   {\it Soit $X$ une vari\'et\'e
projective lisse de dimension $3$ fibr\'ee au-dessus d'une courbe $C$ sur un corps fini $\F$,
la fibre g\'en\'erique \'etant une surface lisse g\'eom\'etriquement int\`egre
 $V$ sur $K=\F(C)$.  Soit $l\neq \car(\F)$ un nombre premier. On suppose que:
 \begin{thlist}
\item 
 L'application cycle
$$ CH^1(X) \otimes \Q_{l} \to H^2_{\et}(X,\Q_{l}(1))$$
est surjective (conjecture de Tate en codimension $1$).
 \item  On a $H^3_\nr(X,\Q_{l}/\Z_{l}(2))=0$. 
 \end{thlist}
 Alors la conjecture \ref{congenzerocycle} vaut pour pour la $K$-vari\'et\'e $V$.
}

\medskip

En combinant ce th\'eor\`eme et 
un r\'esultat r\'ecent de Parimala et Suresh (\cite{parimalasuresh}, Th\'eor\`eme  \ref{H3parimalasuresh} ci-dessous),
 on \'etablit l'\'enonc\'e   suivant (cas particulier de notre
th\'eor\`eme \ref{parasugeneralise}) :

\medskip

\noindent {\bf \textit{Th\'eor\`eme}}.
{\it 
Soient $\F$ un corps fini de caract\'eristique $p$ dif\-f\'e\-rente de $2$,  
et $C,S,X$ des vari\'et\'es projectives, lisses, g\'eom\'etriquement connexes
sur $\F$, de dimensions respectives $1,2,3$, \'equip\'ees de morphismes dominants
$X \to S \to C$, la fibre g\'en\'erique  de $X\to S$ \'etant une conique lisse, 
la fibre  g\'en\'erique  de $S \to C$ \'etant une courbe lisse g\'eom\'etriquement int\`egre.
Supposons satisfaite la conjecture de Tate pour les diviseurs sur la surface $S$.
Soit $V/\F(C)$  la fibre g\'en\'erique
de l'application compos\'ee $X \to C$.

Si la $\F(C)$-surface
$V$ porte une famille de z\'ero-cycles locaux de   degr\'es 1  orthogonale au
groupe de Brauer de $V$,  alors il existe un z\'ero-cycle de degr\'e une puissance de $p$
sur $V$.
}

 \medskip
 
 Parimala et Suresh  \'etablissent en effet que l'on a $H^3_\nr(X,\Q_{l}/\Z_{l}(2))=0$
pour les vari\'et\'es ci-dessus et $l \neq p$. Dans leur article, ils \'etablissent le th\'eor\`eme 
ci-dessus
pour  $S=\P^1_{C}$ : dans ce cas la conjecture de Tate est bien connue.

\medskip

D\'etaillons maintenant le contenu de l'article.

\medskip

Au \S \ref{outils}, on d\'ecrit les outils de cohomologie motivique utilis\'es dans l'article.
 La suite exacte \eqref{eqLK} est particuli\`erement importante. 

\medskip

Au \S \ref{structureH3}, on   donne une nouvelle preuve  
 d'une partie  du th\'eor\`eme  1.1  
 de \cite{kahncycle} sur la structure
du groupe $H^3_\nr(X,\Q/\Z(2))$,  lorsque le corps de base est un corps fini $\F$
(et plus g\'en\'eralement un corps \`a cohomologie galoisienne finie en $l$).

\medskip

Aux \S\S \ref{tailledivisibleH3},    \ref{taillefiniH3} et   \ref{exemplesquestions}, on se penche sur la taille du groupe $H^3_\nr(X,\Q/\Z(2))$ pour
$X$ projective et lisse sur un corps fini.  

\medskip

Au  \S \ref{tailledivisibleH3}, on
 explique, avec un peu plus de d\'etails que dans \cite{kahncycle}, comment  certaines conjectures  
  impliquent sa finitude.
On exhibe
des classes importantes de vari\'et\'es qui v\'erifient cette conclusion:   citons ici la proposition  \ref{surfacedomine} et le th\'eor\`eme  \ref{copiekahn}.

\medskip

Pour $X$ de dimension au plus 2, le groupe $H^3_\nr(X,\Q/\Z(2))$ est nul
(voir  la proposition
\ref{ctssgros}).  On donne au \S  \ref{taillefiniH3} deux cas de
divisibilit\'e
de ce groupe pour $X$ de dimension $3$.

\medskip

\noindent {\textit {\textbf {Question \ref{questiondim3fini}}}}.
Pour   $X/\F$  une vari\'et\'e projective et lisse de dimension~$3$,
a-t-on $H^3_\nr(X,\Q/\Z(2))=0$?

\medskip

  L'analogue de cette   question pour les vari\'et\'es complexes a une r\'eponse n\'egative,
comme l'ont montr\'e de nombreux exemples (cf. \cite[\S 5]{ctv}).

\medskip

 On a    $H^3_\nr(X,\Q/\Z(2))=0$  pour  plusieurs classes de  vari\'et\'es $X$ de dimension 3  g\'eom\'etriquement unir\'egl\'ees
 (\cite{parimalasuresh}, \cite{ctsurvoisin}).
 Par ailleurs,  pour une vari\'et\'e complexe $X$ de dimension 3   unir\'egl\'ee,  on peut gr\^ace \`a
un  th\'eor\`eme de C.~Voisin montrer  $H^3_\nr(X,\Q/\Z(2))=0$ (\cite[Thm. 1.2]{ctv}).
Ceci motive :

\medskip
  
  \noindent{\bf \textit{Conjecture \ref{conjunireglee3finie}}}.
{\it  Pour   $X/\F$  une vari\'et\'e projective et lisse de dimension~$3$ g\'eom\'etriquement unir\'egl\'ee,
 on a $H^3_\nr(X,\Q/\Z(2))=0$.} \\

\medskip

Au \S  \ref{desgal}, 
on fait de la descente galoisienne sur les groupes
 $H^3_\nr(X,\Q/\Z(2))$ et  $CH^2(X)$ 
pour $X$ une vari\'et\'e   projective, lisse, g\'eom\'etriquement connexe sur un corps $F$ \emph{a priori} quelconque.
 On \'etudie en particulier le groupe
 $$\Ker\left(H^{3}_\nr(X,\Q/\Z(2)) \to H^{3}_\nr(\ovX,\Q/\Z(2))\right),$$
o\`u  $\ovX=X\times_{F}F_{s}$ pour une cl\^oture s\'eparable $F_{s}$ de $F$; 
on relie ce noyau \`a la $K$-th\'eorie
 du corps des fonctions de $\ovX$. Lorsque   $F$ est de dimension cohomologique $\leq 1$,
 ceci donne un premier contr\^ole sur
 les noyau et
conoyau de la restriction
$$CH^2(X)\to CH^2(\ovX)^G,$$
o\`u $G=\Gal(F_{s}/F)$ (th\'eor\`eme \ref{cdFleq1}).

 \medskip

Au \S \ref{varfinies}, on applique les r\'esultats   du \S  \ref{desgal}
au cas des vari\'et\'es sur les corps finis.  Des r\'esultats connus 
sur la $K$-cohomologie de telles vari\'et\'es \cite{ctraskind,grossuwa},
et qui reposent sur les conjectures de Weil  d\'emontr\'ees par Deligne,  
nous permettent alors d'\'etablir   le   :
 
\medskip

\noindent{\bf \textit {Th\'eor\`eme \ref{hauptsatz}}}. {\it 
Soit $\F$ un corps fini.
Soit $\ovF$ une cl\^{o}ture alg\'{e}brique de $\F$, et soit $G$ le groupe de Galois de $\ovF$
sur
$\F$. Soit $V$ une $\F$-vari\'{e}t\'{e} projective, lisse, g\'{e}om\'{e}triquement int\`{e}gre.
Soit $M$ le module galoisien fini
$\bigoplus_{l} H^3(\ovV,\Z_{l}(2))\{l\}$.
On a alors une suite exacte 
\begin{multline*}
0 \to \Ker \left(CH^2(V) \to CH^2(\ovV)\right)  \to  H^1(G,M) \\
\to  \Ker\left(H^3_\nr(V,\Q/\Z(2)) \to H^3_\nr(\ovV,\Q/\Z(2))\right)\\
\to   \Coker\left(CH^2(V) \to CH^2(\ovV)^G\right) \to 0.
\end{multline*}

}

 Au \S \ref{globalpositif}, on rappelle les conjectures locales-globales g\'en\'erales de
\cite{ctseattle} sur
 les z\'ero-cycles dans le cas des vari\'et\'es  sur un corps global $K=\F(C)$ de caract\'eristique positive,
 et on montre comment le th\'eor\`eme de Parimala et Suresh \cite{parimalasuresh} permet d'\'etablir
 ces conjectures pour des surfaces  fibr\'ees en coniques au-dessus d'une courbe
 de genre quelconque, lorsque l'on suppose la conjecture de Tate pour
les surfaces sur $\F$
  (Th\'eor\`eme \ref{parasugeneralise}).
 On \'etablit en particulier les deux th\'eor\`emes cit\'es  au d\'ebut de l'introduction.
 
\medskip

 Au \S \ref{global}, nous consid\'erons les  vari\'et\'es  
 projectives et lisses  
sur un corps global quelconque.
Pour le groupe $H^3_{\nr}(V,\Q/\Z(2))$ d'une telle vari\'et\'e $V$, 
nous formulons    
  des  questions et conjectures
   \'etroitement li\'ees d'une part 
  aux conjectures pour les vari\'et\'es sur un corps fini discut\'ees au \S 
\ref{exemplesquestions}, d'autre part
 \`a une conjecture
faite en 1981 par  Sansuc et  le premier auteur  pour  $V$  surface rationnelle sur un corps de nombres  \cite{cts81}. 
  On voit  comment le
th\'eor\`eme  \ref{parasugeneralise}  est parall\`ele aux r\'esultats d\'ej\`a obtenus pour les surfaces fibr\'ees en coniques
au-dessus d'une courbe sur un corps de nombres, sous l'hypoth\`ese que le groupe de Tate-Shafarevich
de la jacobienne de la courbe est fini (cf. \cite{wittenberg}).

\subsection*{Notations}
  \'Etant donn\'e un groupe ab\'elien $A$, un entier $n>0$ et un nombre premier $l$,
 on note $A[n]$ le sous-groupe de $A$ form\'e des \'el\'ements annul\'es par $n$,
 et on note $A\{l\}$ le sous-groupe de torsion $l$-primaire de $A$. On note $A_\div$ son
sous-groupe divisible maximal.

 \section{Les outils}\label{outils}
 
\subsection{Cohomologie non ramifi\'ee} Soient $k$ un corps et $X$ une $k$-vari\'et\'e. Pour
tout nombre premier
$l$ diff\'erent de la caract\'eristique $p$
 de $k$,   tout entier $n>0$, tout entier $j\in \Z$ et tout ouvert $U \subset X$, on dispose
du groupe de cohomologie \'etale
 $H^{i}_{\et}(U,\mu_{l^n}^{\otimes j}).$ 
 En faisceautisant ces groupes
 pour la topologie de Zariski sur $X$, on obtient des faisceaux ${\mathcal H}^{i}_{X}(\mu_{l^n}^{\otimes j}).$
 On peut faire la m\^eme construction en rempla\c cant $\mu_{l^n}^{\otimes j}$ par
$\Q_{l}/\Z_{l}(j)=\colim_{n}\mu_{l^n}^{\otimes j}$.
 
On d\'efinit les groupes de cohomologie non ramifi\'ee $H^{i}_\nr(X,\mu_{l^n}^{\otimes j})$ et
\break
$H^{i}_\nr(X,\Q_{l}/\Z_{l}(j))$ par les formules
$$H^{i}_\nr(X,\mu_{l^n}^{\otimes j}) = H^0(X,{\mathcal H}^{i}_{X}(\mu_{l^n}^{\otimes j}))$$
et
 $$H^{i}_\nr(X,\Q_{l}/\Z_{l}(j)) = H^0(X,{\mathcal H}^{i}_{X}(\Q_{l}/\Z_{l}(j)))).$$
 
 Comme il est expliqu\'e par exemple dans \cite{ctbarbara}, la conjecture de Gersten pour la cohomologie \'etale, \'etablie par
 Bloch et Ogus \cite{bo}, permet pour $X/k$ {\it lisse} et int\`egre, de corps des fonctions $k(X)$,
d'identifier ces groupes \`a
 $$\Ker \big(H^{i}(k(X),\mu_{n}^{\otimes j}) \to \bigoplus_{x \in X^{(1)}} H^{i-1}(k(x),
\mu_{n}^{\otimes j-1})\big)$$
 et
  $$\Ker \big(H^{i}(k(X),\Q_{l}/\Z_{l}(j)) \to \bigoplus_{x \in X^{(1)}} H^{i-1}(k(x),
\Q_{l}/\Z_{l}(j-1))\big),$$
  o\`u $x$ parcourt l'ensemble des points de codimension 1 de $X$, le corps $k(x)$ est
  le corps r\'esiduel en $x$  et les fl\`eches sont des
  applications r\'esidus en cohomologie galoisienne associ\'ees aux anneaux de valuation discr\`ete ${\mathcal O}_{X,x}$.
Ceci permet de montrer  que les groupes  $H^{i}_\nr(X,\mu_{l^n}^{\otimes j})$ et
$H^{i}_\nr(X,\Q_{l}/\Z_{l}(j))$ sont des invariants  $k$-birationnels des $k$-vari\'et\'es int\`egres projectives et lisses.
Pour de telles vari\'et\'es, ceci
permet de les identifier aux groupes de cohomologie non ramifi\'ee introduits par Ojanguren et le premier auteur dans \cite{ct-o}.

Une cons\'equence de la conjecture de Bloch--Kato 
(maintenant un th\'eor\`eme, gr\^ace \`a plusieurs
auteurs, parmi lesquels nous citerons par ordre alphab\'etique  Rost, Suslin, Voevodsky,
Weibel)
est que pour $i \geq 1$  les groupes 
$$H^{i}_\nr(X,\Q_{l}/\Z_{l}(i-1))$$ sont la r\'eunion, et non seulement la limite inductive, des groupes
$H^{i}_\nr(X,\mu_{l^n}^{\otimes i-1}) $.

En utilisant les complexes de de Rham-Witt (Bloch,   Illusie, Milne), pour $X$ r\'egulier de type
fini sur un corps $k$ de caract\'eristique $p>0$ et $i>0$, on d\'efinit des groupes
\begin{align}
H^i(X,\Z/p^n(j))&:= H^{i-j}_\et(X,\nu_n(j))\label{eqHWl}\\ 
H^i(X,\Q_p/\Z_p(j)) &:= \colim\nolimits_n H^i(X,\Z/p^n(j))\notag
\end{align}
o\`u $\nu_n(j)$ est le faisceau de Hodge-Witt logarithmique de poids $j$ et de niveau $n$, d'o\`u des groupes 
$H^{i}_\nr(X,\Q_{p}/\Z_{p}(i-1))$
\cite[d\'ef. 2.7, App. A]{kahncycle}.

Pour $i\geq 1$, et $X$ lisse sur un corps $k$, on d\'efinit :
$$H^{i}_\nr(X,\Q/\Z(i-1)) =  \bigoplus_{l} H^{i}_\nr(X,\Q_{l}/\Z_{l}(i-1)).$$

\subsection{Complexes motiviques} Le pr\'esent article, tout comme l'article \cite{kahncycle},
utilise de fa\c con cruciale les complexes  $\Z(2)$,   version moderne des complexes $\Gamma(2) $ de Lichtenbaum, et leurs propri\'et\'es.
On notera que celles-ci ne font intervenir que le th\'eor\`eme de Merkurjev--Suslin;
la conjecture g\'en\'erale de Bloch--Kato n'est pas utilis\'ee dans  \cite{kahncycle} ni dans le pr\'esent article  --- alors qu'elle
l'est,  en degr\'e 3, dans l'article \cite{ctv}.

Les complexes  $\Gamma(2)$ et $\Z(2)$ sont des objets de la cat\'egorie d\'eriv\'ee des
faisceaux sur le petit site Zariski d'un sch\'ema $X$; on note leurs \'etalifi\'es avec un indice $\et$.
 \'Etant donn\'e une vari\'et\'e $X$ lisse (voire un sch\'ema r\'egulier de type fini) sur un
corps
$k$ d'exposant caract\'eristique
$p$ et un entier $n$ inversible sur
$X$, on a  
des triangles exacts (``suites de Kummer et d'Artin-Schreier'') dans cette cat\'egorie
d\'eriv\'ee :
\begin{equation}\label{LK}
 \Gamma(2)   
\by{\times n} \Gamma(2)  
 \to \mu_{n}^{\otimes 2}\by{+1}, \quad  \Gamma(2)   
\by{\times p^r} \Gamma(2)  
 \to \nu_{r}(2)[-2]\by{+1} 
 \end{equation}
 (Lichtenbaum \cite{lichtenbaum1, lichtenbaum2},  Kahn \cite{kahn}),
et leur variante ``moderne''
\begin{equation} \label{LSVGL2}
\Z_\et(2)  \by{\times n}   \Z_\et(2) \to
\mu_{n}^{\otimes 2} \by{+1} , \; \Z_\et(2)   
\by{\times p^r} \Z_\et(2)  
 \to \nu_{r}(2)[-2]\by{+1} 
  \end{equation}
 Ces derniers sont des  cas particuliers de triangles exacts
\begin{equation} \label{LSVGLq}
   \Z_\et(i) \by{\times n}  \Z_\et(i) \to \mu_{n}^{\otimes i} \by{+1} , \quad  \Z_\et(i)   
\by{\times p^r} \Z_\et(i)  
 \to \nu_{r}(i)[-i]\by{+1}
     \end{equation}
o\`u $\nu_r(i)$ est le faisceau de Hodge-Witt logarithmique de poids $i$ et de niveau $r$ si
$p>1$, et est $0$ par d\'efinition si $p=1$.  On renvoie
\`a \cite[th. 2.6]{kahncycle} pour les d\'etails et les r\'ef\'erences.

\subsection{Cohomologie motivique et $\K$-cohomologie} Certains des groupes d'hypercohomologie
de complexes motiviques utilis\'es dans le pr\'esent article ont \'et\'e \'etudi\'es dans le
pass\'e, en particulier dans l'article \cite{ctraskind},
dans le contexte de la $K$-cohomologie. Ils apparaissent aussi
comme groupes de Chow sup\'e\-rieurs (Bloch).

Rappelons ici que 
l'on note $\K_{i,X}$ 
les faisceaux Zariski sur 
un sch\'ema $X$ associ\'es au pr\'efaisceau qui \`a un ouvert
$U$ associe le groupe $K_{i}(H^0(U,O_{X}))$,
o\`u $K_{i}(A)$ est le $i$-i\`eme groupe de $K$-th\'eorie d'un anneau commutatif $A$, 
comme d\'efini par Quillen.

Pour la commodit\'e du lecteur, nous r\'ep\'etons ici les  identifications
(voir \cite[Thm. 1.1 et 
Thm. 1.6]{kahn} et
\cite[\S 2]{kahncycle}) entre les divers groupes qui nous int\'eressent ici.

\begin{prop}\label{isosbasdegre}
Pour  $X$ une vari\'et\'e lisse sur un corps, on a des isomorphismes canoniques
\begin{align*} CH^2(X,2) \simeq H^0(X,\K_{2})  &\simeq \bH^2_\Zar(X,\Gamma(2))  \simeq
\bH^2_{\et}(X,\Gamma(2))\\  &\simeq  \bH^2_\Zar(X,\Z(2)) \simeq   \bH^2_\et(X,\Z(2));\\
CH^2(X,1) \simeq H^1(X,\K_{2})  &\simeq \bH^3_\Zar(X,\Gamma(2))  \simeq  
\bH^3_{\et}(X,\Gamma(2))\\  &\simeq \bH^3_\Zar(X,\Z(2)) \simeq  \bH^3_\et(X,\Z(2));\\
CH^2(X) = CH^2(X,0) \simeq H^2(X,\K_{2}) &\simeq   \bH^4_\Zar(X,\Gamma(2))\\
&\simeq   \bH^4_\Zar(X,\Z(2)).
\end{align*}
\end{prop}

\subsection{$CH^2$ et $H^3_\nr$} Notons que dans la proposition ci-dessus, les isomorphismes
entre groupes de cohomologie Zariski et \'etale s'arr\^etent en degr\'e $3$. Outre le triangle
exact
\eqref{LSVGL2},  
 la suite exacte 
\begin{equation}\label{eqLK}
0 \to CH^2(X) \to \bH^4_{\et}(X,\Z(2)) \to H^3_\nr(X,\Q/\Z(2))  \to 0
\end{equation}
joue un r\^ole central dans le pr\'esent article. Avec le complexe $\Gamma(2)$ de Lichtenbaum \cite{lichtenbaum1}
\`a la place de  de $\Z(2)$,
cette suite   appara\^{\i}t pour
la premi\`ere fois (mais de fa\c con pas compl\`etement  explicite)  
 dans   un article de Lichtenbaum (\cite{lichtenbaum2}, Thm. 2.13, sa d\'e\-mons\-tra\-tion, et
remarque  2.14),
 o\`u elle est \'etablie \`a la $2$-torsion pr\`es.
Toujours avec le complexe $\Gamma(2)$,
elle   est  \'etablie en toute g\'en\'eralit\'e dans
\cite[Thm. 1.1]{kahn}.
  Sous la forme moderne ci-dessus, la suite est \'etablie dans \cite[Prop. 2.9]{kahncycle}.

\subsection{Descente} Enfin on fait de la descente galoisienne sur l'hypercohomologie des
complexes 
$\Z(2)$.
Dans \cite{ctv}, on fait cela sur les complexes de Gersten en $K$-th\'eorie, comme cela
avait \'et\'e fait par Raskind et le premier auteur dans \cite{ctraskind}, \`a la suite
de Spencer Bloch.
Dans divers articles, et en particulier \cite{kahn2, kahn}, 
le second auteur a montr\'e l'utilit\'e de la descente galoisienne
sur l'hypercohomologie \'etale des complexes $\Gamma(2)$ ou $\Z(2)$.

\section{Cycles de codimension $2$ et   $H^3$ non ramifi\'e}
 \label{structureH3}

Soit $l$ un nombre premier. On dit qu'un corps $k$ est \emph{\`a cohomologie galoisienne finie en $l$} 
  si  $l\neq \car(k)$ et si, pour
tout module galoisien fini 
$M$  d'ordre une puissance de $l$
et tout  entier $i\geq 0$, les groupes de cohomologie
$H^{i}(k,M)$ sont finis. Ceci implique que pour toute $k$-vari\'et\'e
$X$ les groupes de cohomologie \'etale $H^{i}_{\et}(X,M)$ sont finis.
Exemples de tels corps : les corps alg\'ebri\-que\-ment clos, les corps
finis, le corps des r\'eels, les corps $p$-adiques, les corps locaux sup\'erieurs
dont la derni\`ere caract\'eristique r\'esiduelle est distincte de $l$.

La consid\'eration de tels corps simplifie l'exposition de l'application cycle, et la d\'emonstration du th\'eor\`eme principal de cette section.

Le lemme suivant est bien connu et tr\`es utile:

\begin{lem}\label{l2.1}
 Soient $A$ un groupe ab\'elien et $l$ un nombre premier.
On a une application naturelle
\begin{equation}\label{eq2.0}
A \otimes \Z_{l} \to   \lim\nolimits_{n} A/l^n.
\end{equation}
\begin{thlist}
\item Pour tout $n>0$, cette application induit modulo $l^n$
un isomorphisme
$$(A \otimes \Z_{l} )/l^n \iso A/l^n.$$
\item Le $\Z_{l}$-module $ \lim\nolimits_{n} A/l^n$  est de type fini si et seulement si $A/l$ est fini.
Si c'est le cas, 
 la fl\`eche \eqref{eq2.0}
est surjective.
\item Soit $T$ un groupe ab\'elien de torsion.
Le groupe $T \otimes \Z_{l}  $ s'identifie naturellement
au sous-groupe de torsion $l$-primaire de $T$. 
\end{thlist}
\end{lem}

\begin{proof}
D\'emontrons simplement (ii). Soit $B \subset A$ un sous-groupe de type fini engendrant $A$ modulo $l$.
On v\'erifie par r\'ecurrence sur $n$ que les applications de groupes finis $B/l^n \by{f_{n}} A/l^n$ sont surjectives. D'o\`u la surjectivit\'e des $\Ker f_{n+1}\to \Ker f_{n}$, qui entra\^\i ne la deuxi\`eme des surjections
$$B\otimes \Z_{l} \Surj   \lim\nolimits_{n} B/l^n \Surj   \lim\nolimits_{n} A/l^n.$$
 \end{proof}

\subsection{L'application cycle} Soit $X$ une $k$-vari\'et\'e lisse, et soit $l$ un nombre premier. 

Supposons d'abord $l$ distinct de la caract\'eristique de $k$. Puisque $k$ est \`a cohomologie galoisienne finie en $l$, le groupe
\begin{equation}\label{eql-adic}
H^j_\et(X,\Z_l(i))=\lim\nolimits_n H^j_\et(X,\mu_{l^n}^{\otimes i})
\end{equation}
est un $\Z_l$-module de type fini (lemme \ref{l2.1}).

Supposons maintenant $k$ \emph{fini} de caract\'eristique $p$ et $X$ \emph{projective}. Alors les groupes $H^j_\et(X,\Z/p^n(i))$ de \eqref{eqHWl} sont \emph{finis} \cite[cor. 1.12]{milneamer}, et on peut d\'efinir $H^j_\et(X,\Z_p(i))$ comme en \eqref{eql-adic}: c'est un $\Z_p$-module de type fini.

Dans les deux cas cit\'es, on dispose d'applications classe de cycle
\begin{equation}\label{eqtate}
CH^{i}(X) \otimes \Z_{l} \to H^{2i}_{\et}(X,\Z_{l}(i)) \quad (i\ge 0)
\end{equation}
d\'efinies comme limites projectives des applications classe de cycle usuelles:
\[CH^{i}(X) /l^n\to H^{2i}_{\et}(X,\Z/l^n(i)).\]

Pour $l\ne p$, celles-ci sont construites dans \cite{cycle}. Pour $l=p$, une construction est donn\'ee dans \cite{gros}, et une construction \'equivalente dans \cite{milneamer}. Une autre est donn\'ee dans \cite[\S 3.1]{kahncycle}; il est probable, mais non d\'emontr\'e, que les deux constructions co\"\i ncident.

 \subsection{\'Enonc\'e du th\'eor\`eme}
 
\begin{theo}\label{Hauptsatz} Soient $l$ un nombre premier, $k$ un corps  \`a cohomologie galoisienne finie en $l$, et soit $X$ une
$k$-vari\'et\'e lisse. Soit  $M$ le
$\Z_{l}$-module de type fini
$$M = \Coker \left(CH^2(X) \otimes \Z_{l} \to H^4_{\et}(X,\Z_{l}(2))\right).$$
Le quotient de
$H^3_\nr (X,\Q_{l}/\Z_{l}(2))$
 par son sous-groupe divisible maximal
 est le groupe fini $M_\tors$.\\
 Si $k$ est fini et $X$ est projective, cette assertion s'\'etend \`a $l=p$, caract\'eristique de $k$.
\end{theo}

Ce th\'eor\`eme est un cas particulier d'un th\'eor\`eme du second auteur \cite[th. 1.1]{kahncycle}
valable sur un corps quelconque, les classes de cycle \'etant \`a valeurs dans la cohomologie \'etale continue de Jannsen.  Le cas o\`u le corps de base est le corps des complexes
avait \'et\'e consid\'er\'e par C.~Voisin et le premier auteur dans \cite{ctv}.

La m\'ethode que nous utilisons ici, sp\'ecifique au cas des corps
\`a cohomologie galoisienne finie, repose sur un argument de comptage.
Elle   a \'et\'e reprise par Alena Pirutka
\cite{pirutka3}
pour \'etudier, pour tout $i$ entier,  le conoyau de l'application
$$\Coker \left(CH^{i}(X) \otimes \Z_{l} \to H^{2i}_{\et}(X,\Z_{l}(2))\right)$$
pour $X$ projective et lisse de dimension $d$ sur un corps $k$  \`a cohomologie galoisienne finie en $l$,
avec application particuli\`ere au cas $i=d-1$ lorsque $k$ est un corps fini.

 \subsection{D\'emonstration du th\'eor\`eme} Cette d\'emonstration fonctionne uniform\'ement pour $l\ne p$ ou pour $l=p$, en prenant la d\'efinition de la classe de cycle $p$-adique  de \cite{kahncycle}.

On consid\`ere la suite exacte fondamentale (\ref{eqLK})
$$0 \to CH^2(X) \to \bH^4_\et(X,\Z(2)) \to H^3_\nr (X,\Q/\Z(2)) \to 0.$$
On la tensorise avec $\Z_{l}$. Cela donne  
$$0 \to CH^2(X)\otimes \Z_{l } \to \bH^4_\et(X,\Z(2))\otimes \Z_{l } \to H^3_\nr (X,\Q_{l}/\Z_{l}(2)) \to 0.$$

Du triangle exact  (\ref{LSVGL2})
on d\'eduit  
 la suite exacte de groupes ab\'eliens
\begin{equation}\label{eq2.4}
0 \to \bH^4_\et(X,\Z(2))/l^n \to  H^4_{\et}(X,\mu_{l^{n}}^{\otimes 2}) \to
\bH^5_\et(X,\Z(2))[l^n] \to 0.
\end{equation}

Les fl\`eches
$$\bH^4_\et(X,\Z(2)) \to \bH^4_\et(X,\Z(2))/l^n \to H^4_{\et}(X,\mu_{l^{n}}^{\otimes 2}     )   $$
induisent des fl\`eches
$$\bH^4_\et(X,\Z(2)) \to  \lim_{n} \bH^4_\et(X,\Z(2))/l^n \to  \lim_{n} H^4_{\et}(X,\mu_{l^{n}}^{\otimes 2})=H^4_{\et}(X,\Z_{l}(2))$$
(o\`u la derni\`ere \'egalit\'e est ici une d\'efinition) et donc une fl\`eche
\begin{equation}\label{eq2.5}
 \theta_{1} : \bH^4_\et(X,\Z(2)) \otimes
\Z_{l} \to H^4_{\et}(X,\Z_{l}(2)).
\end{equation}

De \eqref{eq2.4} et de la finitude de $H^4_{\et}(X,\mu_{l^{n}}^{\otimes 2})$, on d\'eduit la 
suite  exacte de $\Z_{l}$-modules de type fini
$$0 \to  \lim_{n} \bH^4_\et(X,\Z(2))/l^n \to   H^4_{\et}(X,\Z_{l}(2))  \to   \lim_{n}  
\bH^5_\et(X,\Z(2))[l^n] \to 0.$$

Le terme de droite est sans torsion, car c'est un module de Tate.
 Comme   $ \lim_{n} \bH^4_\et(X,\Z(2))/l^n  $ est un $\Z_{l}$-module de type
 fini, 
 il r\'esulte  alors du lemme \ref{l2.1} (ii)
 que l'application $\theta_{1} $ de (\ref{eq2.5}) 
 a son conoyau sans torsion.

Si $l\ne p$, l'application compos\'ee 
$$CH^2(X)\otimes \Z_{l }  \to \bH^4_\et(X,\Z(2))\otimes
\Z_{l } \to H^4_\et(X,\Z_{l}(2))$$ 
est   l'application cycle classique \cite[Rem.
3.1]{kahncycle}. Si $l=p$, il en est de m\^eme avec la d\'efinition de \cite{kahncycle} (ibid.).

Consid\'erons alors le diagramme commutatif de suites exactes

\begin{equation}\label{eq2.6}
\begin{CD}
 0  & \to & CH^2(X)\otimes \Z_{l } & \to & \bH^4_\et(X,\Z(2))\otimes \Z_{l }  & \to & H^3_\nr (X,\Q_{l}/\Z_{l}(2)) & \to & 0 \\
&&   ||&& 
@VV{\theta_1}V @VV{\theta_{0}}V\\
  &   & CH^2(X)\otimes \Z_{l}  & \to & H^4_\et(X,\Z_{l}(2)) & \to & M & \to  & 0.
\end{CD}
\end{equation} 
o\`u la fl\`eche  $\theta_{0}$ 
est induite par le reste du diagramme. L'image de cette fl\`eche est \'evidemment de torsion. Notons
$$\theta : H^3_\nr (X,\Q_{l}/\Z_{l}(2))  \to M_\tors$$
l'application induite.
On a vu que la fl\`eche verticale m\'ediane a son conoyau sans torsion, et le diagramme montre que ce conoyau s'identifie \`a celui de 
$\theta_{0}$.
 Ainsi $\theta$ est surjective.

Par d\'efinition de la fl\`eche $\bH^4_\et(X,\Z(2))\otimes \Z_{l } \to H^4_\et(X,\Z_{l}(2)) $,
pour tout $n>0$, 
le diagramme suivant commute
\[\begin{CD}
(\bH^4_\et(X,\Z(2))\otimes \Z_{l })/l^n  @>{a}>> \bH^4_\et(X,\Z(2))/l^n \\
 @V{c}VV @V{d}VV\\
 H^4_\et(X,\Z_{l}(2))/l^n @>{b}>>  H^4_{\et}(X,\mu_{l^{n}}^{\otimes 2}) 
 \end{CD}\] 

La fl\`eche $a$
est un isomorphisme (lemme \ref{l2.1} (i)).
La fl\`eche $d$
est une injection, d'apr\`es (\ref{LSVGL2}).
On conclut que $c$
est injectif.

Notons $K=H^3_\nr (X,\Q_{l}/\Z_{l}(2))$. Passant au quotient modulo $l^n$ dans le diagramme
\eqref{eq2.6}, on voit que $c$
induit une injection
$$ \theta_{0} \hskip2mm {\rm mod} \hskip1mm  l^n \hskip1mm : K/l^n \hookrightarrow M/l^n.$$
On a vu que $\theta_{0}$ induit une surjection $
\theta : K \to M_\tors$.
On conclut que $\theta$ induit un isomorphisme 
$$ \theta \hskip2mm {\rm mod} \hskip1mm  l^n \hskip1mm :  K/l^n \to M_\tors/l^n.$$

Le groupe $ M_\tors$ est fini, en particulier d'exposant $l^e$ fini. Pour $n \geq e$,
la projection naturelle $M_\tors/l^{n+1} \to M_\tors/l^n$ est donc un isomorphisme.
Il en r\'esulte que pour $n \geq e$, les fl\`eches 
$$K/l^{n+1} \to K/l^{n} \to M_\tors/l^n$$
et $M_\tors \to M_\tors/l^n$ 
sont des isomorphismes. Ainsi pout tout tel $n$, $l^nK=l^{n+1}K$.
Le sous-groupe $l^eK \subset K$ est donc un groupe divisible,
  c'est le sous-groupe divisible maximal $K_\div$ de $K$, et le quotient de
  $K$ par $K_\div$ s'identifie au groupe fini $M_\tors$.

On a donc montr\'e que le quotient
$K=H^3_\nr(X,\Q_{l}/\Z_{l}(2))$ par son sous-groupe divisible maximal 
s'identifie  au sous-groupe de torsion du conoyau de l'application cycle
$$CH^2(X)\otimes \Z_{l}   \to  H^4_\et(X,\Z_{l}(2)),$$
sous-groupe de torsion  qui est fini.

\subsection{Remarques}

  \begin{rema} On a vu dans la d\'emonstration que la fl\`eche $\theta_1$ de \eqref{eq2.5} a un
conoyau sans torsion. On en d\'eduit l'\'enonc\'e suivant: \emph{toute classe
$\alpha\in H^4_{\et}(X,\Z_{l}(2) )_\tors$ est dans l'image de
$\bH^4_\et(X,\Z(2))\otimes \Z_l$}. Dans
\cite[cor. 3.5]{kahncycle}, il est d\'emontr\'e que le noyau de $\theta_1$ est divisible, 
ce qui implique que $\alpha$ provient m\^eme de 
$\bH^4_\et(X,\Z(2))_\tors\iso \left(\bH^4_\et(X,\Z(2))\otimes\Z_l\right)_\tors$ (voir lemme
\ref{l2.1} (iii)), donc en particulier de $\bH^4_\et(X,\Z(2))$.

Si $k=\C$, si $X$ est projective et si on remplace $\Z(2)$ par $\Gamma(2)$, ce dernier 
\'enonc\'e est d\^u \`a Lichtenbaum  \cite[th. 2.15]{lichtenbaum2}. L'argument
ci-dessus \'evite le th\'eor\`eme de Lieberman, utilis\'e dans \cite{lichtenbaum2}, selon lequel l'\'equivalence
homologique co\"\i ncide avec l'\'equivalence num\'erique sur les cycles de codimension $2$
(th\'eor\`eme reposant
\emph{in fine} sur le th\'eor\`eme (1,1) de Lefschetz).

 \begin{comment} Soit $k=\bf C$ le corps des complexes. Pour tout $l$ premier, 
on a le diagramme commutatif de suites exactes
   \[\begin{CD}
 0  & \to & CH^2(X)  & \to & \bH^4_\et(X,\Z(2))  & \to & H^3_\nr (X,\Q/\Z(2)) & \to & 0 \\
&&   @VVV @VVV @VV{\theta_{0}}V\\
  &   & CH^2(X)\otimes \Z_{l}  & \to & H^4_\et(X,\Z_{l}(2)) & \to & M & \to  & 0.
\end{CD}\] 
On a \'etabli la surjectivit\'e de la fl\`eche induite $$\theta : H^3_\nr (X,\Q/\Z(2)) \to M_\tors.$$
 
Ceci est \`a comparer 
avec le r\'esultat suivant. Pour une vari\'et\'e projective et lisse
sur les complexes, Lichtenbaum  (\cite{lichtenbaum2}, Th\'eor\`eme 2.15, dont la d\'emonstration 
repose sur
un r\'esultat de Lieberman) 
montre que toute classe
dans $H^4_{\et}(X,\Z_{l}(2) )_\tors$ est dans l'image de $\bH^4_\et(X,\Z(2))$.
\end{comment}

Si l'on lit cet \'enonc\'e \`a la lueur de la suite exacte
$$0 \to CH^2(X) \to \bH^4_\et(X,\Z(2)) \to H^3_\nr (X,\Q/\Z(2)) \to 0,$$
on voit que le corollaire 2.16 de \cite{lichtenbaum2}  implique  que les contre-exemples de Atiyah et Hirzebruch 
\`a la conjecture de Hodge enti\`ere pour les cycles de codimension 2
donnent des exemples sur $\C$ avec $H^3_\nr (X,\Q/\Z(2)) \neq 0$
(voir  \cite{ctv}).
\end{rema}

\begin{rema}
Soit $X$ une vari\'et\'e projective, lisse, g\'e\-o\-m\'e\-tri\-que\-ment
con\-nexe sur un corps fini $\F$. Supposons $H^3_\nr(X,\Q_{l}/\Z_{l}(2))$ fini
(comme on le verra au \S  \ref{tailledivisibleH3}, c'est 
peut-\^etre toujours le cas).
La proposition 4.3.5
de Saito-Sato  \cite{saitosato} \'etablit  
alors 
le r\'esultat suivant, cas particulier
du Th\'eor\`eme \ref{Hauptsatz} :
  {\it l'ordre} de $H^3_\nr(X,\Q_{l}/\Z_{l}(2))$ est \'egal \`a {\it l'ordre} du sous-groupe de torsion du groupe
$$\Coker \left(CH^2(X) \otimes \Z_{l} \to H^4_{\et}(X,\Z_{l}(2))\right).$$
\end{rema}

 \section{Finitude de
$H^3$ non ramifi\'e  : r\'esultats et conjectures}\label{tailledivisibleH3}

 Commen\c cons
  par un rappel.
   
 \begin{prop}\label{ctssgros} Soit $k$ un corps alg\'ebriquement clos.
 Pour toute $k$-vari\'et\'e lisse $X$
 de dimension $d \leq 2$ (resp. pour tout corps de fonctions $K/k$ en au plus $2$ variables), on
a $H^3_\nr(X,\Q/\Z(2))=0$ (resp. $H^3_\nr(K/k,\Q/\Z(2))=0$). La m\^eme conclusion est vraie si
$k$ est fini et $X$ \emph{projective}. 
 \end{prop}
 
 Pour
 la torsion premi\`ere \`a $p={\rm car}(k)$,  le r\'esultat est
clair sur un corps s\'eparablement clos (la dimension cohomologique
 du corps de fonctions $k(X)$ est au plus $2$);  sur un corps fini, il est 
 d\^u \`a Sansuc, Soul\'e et au premier auteur \cite[Rem. 2 p.~790]{ctss}.
 Pour la $p$-torsion,  le r\'esultat sur un corps alg\'ebriquement clos est d\^u \`a N. Suwa	
\cite[lemma 2.1]{suwa} et sur un corps fini \`a K. Kato \cite[cor. p. 145]{katocrelle}, qui se
repose sur des r\'esultats de M. Gros 
\cite{gros0}.  Enfin, le passage des vari\'et\'es projectives lisses
aux corps de fonctions d\'ecoule de la r\'esolution des singularit\'es en dimension $\le 2$  (Abhyankar)
gr\^ace \`a 
\cite[prop. 2.1.8]{ctbarbara}.

 \subsection{Correspondances et motifs birationnels}\label{cormotbir}

 \begin{prop} \label{surfacedomine}
Soit  $k$ un corps fini ou un corps s\'eparablement clos. 
Soit $X$ une $k$-vari\'et\'e projective, lisse, g\'e\-o\-m\'e\-tri\-que\-ment int\`egre. 
Supposons qu'il existe un domaine universel $\Omega$  contenant $k$, 
une surface $Y$ projective et lisse sur $k$ 
 et un $k$-morphisme $f:Y \to
X$
qui induit une surjection
\[f_*:CH_0(Y_\Omega)\otimes\Q \Surj CH_0(X_\Omega)\otimes\Q. \]
Alors pour tout $l$ premier distinct 
de la caract\'eristique de $k$, le groupe $$H^3_\nr(X,\Q_{l}/\Z_{l}(2))$$ est fini,
et il est nul pour presque tout premier $l$.
 \end{prop}

\begin{proof} Soit $d$ la dimension  de $X$. Soit $\Delta_{X} \subset X \times_{k}X$ la diagonale.
Par un argument utilis\'e par S. Bloch \cite[Lecture 1, Appendix]{blochbook}, l'hypoth\`ese assure l'existence
d'un entier $N>0$ tel que la classe de  $N.\Delta_{X} \in CH_{d}(X \times_{k}X)$
soit la somme  de la classe d'un cycle  $Z_{1}   \in g_{*}(CH_{d}(Y\times_{k}X))$,
 et d'un cycle $Z_{2} \in  CH_{d}(X \times_{k}X)$
de restriction nulle \`a $CH_{d}(X \times_{k}U)$ pour $U$ un ouvert de Zariski non vide de $X$.

 (D\'etaillons l\'eg\`erement cet argument. Soit $K$ le corps des fonctions de $X$, qu'on suppose plong\'e dans $\Omega$. Un argument de transfert montre facilement que l'hypoth\`ese implique la surjectivit\'e de l'homomorphisme
\[CH_0(Y_K)\otimes \Q\by{f_*} CH_0(X_K)\otimes \Q.\]

Si $z$ est un $0$-cycle de $Y_K$ \`a coefficients rationnels qui s'envoie sur la classe du point g\'en\'erique de $X$, un rel\`evement de $z$ \`a $Y\times X$ fournit le cycle $Z_1$ cherch\'e.)
  
Les groupes de cohomologie galoisienne $H^{i}_{\et}(\bullet,\Q_{l}/\Z_{l}(j))$ 
associ\'es aux corps contenant $k$
  satisfont les axiomes des modules de cycles de Rost \cite[(1.11)]{rost}.
 Avec les notations de \cite{rost} et \cite{merkurjev}, le groupe $A^0(X,M)$
 associ\'e \`a $M=H^{i}_{\et}(\bullet,\Q_{l}/\Z_{l}(j))$   sur une $k$-vari\'et\'e
 projective et lisse $X$ est le groupe $H^{i}_\nr(X,\Q_{l}/\Z_{l}(j))$.
 
 On dispose donc d'un diagramme commutatif 
\begin{comment}
 \[\begin{CD}
 CH_{d}(Y \times_{k}X)  &  \times  & H^3_\nr(Y,\Q_{l}/\Z_{l}(2)) @>>>
H^3_\nr(X,\Q_{l}/\Z_{l}(2)) \\
   @AA{=}A @AAA @AA{=}A  \\
CH_{d}(Y \times_{k}X)  &  \times  & H^3_\nr(X,\Q_{l}/\Z_{l}(2) )  &  &    H^3_\nr(X,\Q_{l}/\Z_{l}(2))   \\
   @VVV @VV{=}V @VV{=}V  \\
   CH_{d}(X \times_{k}X)  &  \times  & H^3_\nr(X,\Q_{l}/\Z_{l}(2)) @>>>
H^3_\nr(X,\Q_{l}/\Z_{l}(2)) 
 \end{CD}\] 
\end{comment}
\[\xymatrix{
H^3_\nr(Y,\Q_{l}/\Z_{l}(2)) \times  CH_{d}(Y \times_{k}X)     \ar[dr]
\\
 H^3_\nr(X,\Q_{l}/\Z_{l}(2) ) \times CH_{d}(Y \times_{k}X)     \ar[u]^{f^* \times
1}\ar[d]_{1 \times g_{*}}    &    H^3_\nr(X,\Q_{l}/\Z_{l}(2)).   \\
 H^3_\nr(X,\Q_{l}/\Z_{l}(2)) \times  CH_{d}(X \times_{k}X)     \ar[ur]
}\]

 Dans ce diagramme,  les fl\`eches diagonales sont des accouplements de correspondances  
 fournis par la  th\'eorie de cycles de Rost \cite{merkurjev, deglise}. 
 La commutativit\'e du diagramme r\'esulte d'un cas particulier d'une formule de projection
mentionn\'ee dans
  \cite[\S 1]{merkurjev} et \'etablie par F. D\'eglise  \cite[Prop. 5.9 (3)]{deglise}.

L'\'egalit\'e  $$N.\Delta_{X} =Z_{1} + Z_{2} \in CH_{d}(X \times_{k}X) $$
permet de d\'ecomposer la multiplication par $N$ sur   $H^3_\nr(X,\Q_{l}/\Z_{l}(2))$
en la somme de $Z_{1, *}$ et $Z_{2,*}$. D'apr\`es \cite[Prop. 3.1]{merkurjev}, 
$Z_{2,*}=0$.

 La nullit\'e de $H^3_\nr(Y,\Q_{l}/\Z_{l}(2))$  (Proposition \ref{ctssgros})
et le diagramme ci-dessus donnent  $Z_{1,*}=0$.
 
 On voit donc que l'entier $N$ annule $H^3_\nr(X,\Q_{l}/\Z_{l}(2))$. D'apr\`es le
th\'eo\-r\`eme \ref{Hauptsatz}, le groupe   $H^3_\nr(X,\Q_{l}/\Z_{l}(2))$ est extension d'un
groupe fini par un groupe $l$-divisible. Ceci ach\`eve la d\'emonstration.
\end{proof}

 \begin{rema}\label{remsurfdomine2} 
L'hypoth\`ese de la proposition   \ref{surfacedomine} est satisfaite si  
la vari\'et\'e est rationnellement domin\'ee par le produit d'une surface et d'un espace projectif. 
Un autre exemple non trivial est celui de Bloch--Srinivas,  \cite[ex. 1  p. 1242]{bloch-sri}. Si $k$ est de caract\'eristique z\'ero, l'hypoth\`ese est conjecturalement tr\`es proche de la nullit\'e des groupes de cohomologie
coh\'erente $H^{i}(X,{\mathcal O}_{X})$ pour $i \geq 3$.
 \end{rema}

{\bf D\'emonstration alternative de la proposition \ref{surfacedomine}.}

\begin{lem}\label{l4.2} Soit  $\sA$ le quotient de la cat\'egorie des
groupes ab\'eliens par la classe de Serre des groupes ab\'eliens d'exposant fini.
Pour $l\ne \car k$, consid\'erons
$H^3_\nr(-,\Q_l/\Z_l(2))$ comme un foncteur de la cat\'egorie des vari\'et\'es projectives
lisses vers $\sA$. Alors:
\begin{itemize}
\item[(a)] Ce foncteur s'\'etend en un foncteur sur la cat\'egorie $\sM^\eff(k,\Q)$ des
motifs de Chow effectifs \`a coefficients rationnels.
\item[(b)]Ce nouveau foncteur passe au quotient \`a travers la cat\'egorie $\sM^\o(k,\Q)$ des
motifs de Chow birationnels.
\end{itemize}
\end{lem}

\begin{proof} (a) La cat\'egorie $\sA$ est $\Q$-lin\'eaire et 
ab\'elienne: il suffit donc de montrer que les correspondances de Chow \`a coefficients
entiers op\`erent sur ce foncteur. Ce fait est justifi\'e ci-dessus.

(b) Une d\'efinition de $\sM^\o(k,\Q)$ est: l'enveloppe pseudo-ab\'elienne de la localisation
de $\sM^\eff(k,\Q)$ obtenue en inversant les morphismes birationnels \cite[cor.
2.4.2]{kahn-suj}. Il suffit donc de voir que tout morphisme birationnel induit un isomorphisme
sur $H^3$ non ramifi\'e, ce qui est bien connu.
\end{proof}

\begin{lem}\label{l4.1} Soient $X,Y$ deux vari\'et\'es projectives lisses sur un corps $k$ et
soit
$\gamma\in CH_{\dim Y}(Y\times X)\otimes\Q$ une correspondance de Chow. Supposons que
l'application induite
\[CH_0(Y_\Omega)\otimes \Q\to CH_0(X_\Omega)\otimes \Q\]
soit surjective.
 Alors le motif birationnel de $X$ est facteur direct de celui de $Y$.
\end{lem}

\begin{proof} C'est une variante de celle de \cite[prop. 7.6]{qjpam}.
\end{proof}

\begin{prop}\label{p4.1} Notons $(\Q/\Z)'(2)=\bigoplus_{l\ne \car k} \Q_l/\Z_l(2)$. Alors, sous
les hypoth\`eses du lemme
\ref{l4.1}:\\ 
a)  Dans la cat\'egorie $\sA$, le groupe $H^3_\nr(X,(\Q/\Z)'(2))$ est facteur direct du groupe
$H^3_\nr(Y,(\Q/\Z)'(2))$.\\  
b) Si $k$ est \`a cohomologie galoisienne finie en $l$ pour tout $l\ne \car k$, la finitude de $H^3_\nr(Y,(\Q/\Z)'(2))$
implique celle de $H^3_\nr(X,(\Q/\Z)'(2))$ (dans la cat\'egorie des groupes ab\'eliens).
\end{prop}

\begin{proof} a) r\'esulte des lemmes 
\ref{l4.2} et \ref{l4.1}. b) D'apr\`es a), le groupe ab\'elien $H^3_\nr(X,(\Q/\Z)'(2))$ est
d'exposant fini. Mais d'apr\`es \cite[Thm. 1.1]{kahncycle} ou  encore d'apr\`es
le th\'eor\`eme  \ref{Hauptsatz},
pour tout $l$ premier,
 ses composantes
$l$-primaires sont extensions d'un groupe fini par un groupe divisible. Elles sont donc finies,
et au total $H^3_\nr(X,(\Q/\Z)'(2))$ est fini.
\end{proof}

\medskip

Les propositions    \ref{ctssgros} et \ref{p4.1}  donnent imm\'ediatement
la proposition  \ref{surfacedomine}.
\qed

\begin{rema} Si $l=\car k$, les arguments ci-dessus fonctionnent \`a condition de savoir que les
correspondances de Chow \`a coefficients entiers op\`erent sur $H^3_\nr(-,\Q_p/\Z_p(2))$. La
technique de d\'emonstration pr\'ec\'edente ne s'applique plus, car
$K\mapsto H^{*+1}_\et(K,\Q_p/\Z_p(*)) = H^1_\et(K,\nu_\infty(*))$ ne v\'erifie pas les axiomes
des modules de cycles de Rost \`a cause du d\'efaut de puret\'e de la cohomologie de
Hodge-Witt logarithmique. On peut toutefois utiliser la suite exacte \eqref{eqLK}, \`a
condition de savoir que les correspondances de Chow op\`erent sur les groupes de Chow
sup\'erieurs \'etales. Ce fait est r\'edig\'e dans \cite[app. A]{cell0}, sous r\'eserve de
v\'erification de certaines fonctorialit\'es des groupes de Chow sup\'erieurs (sans la
topologie \'etale): ibid., A.1.

Par ailleurs, dans la proposition \ref{p4.1} b), il faut exiger que $k$ soit fini si on veut la
finitude de $H^3_\nr(X,\Q_p/\Z_p(2))$ en conclusion.
\end{rema}

\subsection{Une cons\'equence d'une conjecture de Bass}\label{conjbass}

\medskip

Soit $X$ lisse sur un corps $k$, et  soit $l$ un nombre premier distinct de la  caract\'eristique de $k$.
 D'apr\`es Bloch et Ogus \cite{bo}, on a les suites exactes
 \begin{multline*}0 \to H^1(X,\H^2(\mu_{l^n}^{\otimes 2})) \to H^3_{\et}(X,\mu_{l^n}^{\otimes
2})\\
  \to H^3_\nr(X,\mu_{l^n}^{\otimes 2}) \to CH^2(X)/l^n  \to H^4_{\et}(X,\mu_{l^n}^{\otimes 2})
\end{multline*}
 et par passage \`a la limite sur $n$
\begin{multline*}0 \to H^1(X,\H^2(\Q_{l}/\Z_{l}(2))) \to  H^3_{\et}(X,\Q_{l}/\Z_{l}(2)) \\
 \to
  H^3_\nr(X,\Q_{l}/\Z_{l}(2)) \to CH^2(X)\otimes \Q_{l} /\Z_{l} \to
H^4_{\et}(X,\Q_{l}/\Z_{l}(2)).
\end{multline*}
 
Soit $k=\F$ un corps fini. Supposons en outre $X/\F$ projective, lisse, g\'eom\'e\-tri\-quement
int\`egre.
On sait (ceci utilise  le th\'eor\`eme de Merkurjev-Suslin et
les conjectures de Weil, voir par exemple \cite{ctcime})
 qu'alors la derni\`ere suite se r\'e\'ecrit
\begin{multline*}
0 \to CH^2(X)\{l\} \to H^3_{\et}(X,\Q_{l}/\Z_{l}(2))\\
\to H^3_\nr(X,\Q_{l}/\Z_{l}(2)) \to CH^2(X)\otimes \Q_{l} /\Z_{l} \to
H^4_{\et}(X,\Q_{l}/\Z_{l}(2))
\end{multline*}
et l'on sait  (voir \cite{ctss}) que le groupe $H^3_{\et}(X,\Q_{l}/\Z_{l}(2))$ est fini, et nul pour presque tout
$l$.

 \medskip
 
L'\'enonc\'e suivant est donc connu depuis les ann\'ees 1980.

\begin{prop}
Soit $k=\F$ un corps fini, $p={\rm car}(\F)$.  Soit $X/\F$ projective, lisse, g\'eom\'etriquement int\`egre.
Soit $l \neq p$ un nombre premier.
\begin{itemize}
\item[(a)] Si le groupe $CH^2(X)$ est un groupe de type fini, alors  
le groupe $$H^3_\nr(X,\Q_{l}/\Z_{l}(2))$$ est de cotype fini.
\item[(b)]  Le groupe $CH^2(X)/l^n$ est fini si et seulement si le groupe
$H^3_\nr(X,\mu_{l^n}^{\otimes 2})$  est fini.
\item[(c)] Si le groupe $H^3_\nr(X,\Q_{l}/\Z_{l}(2))$ est de cotype fini,
alors pour tout entier $n>0$, les groupes  $H^3_\nr(X,\mu_{l^n}^{\otimes 2})$ 
et $CH^2(X)/l^n$ sont finis.
\item[(d)] Le groupe $H^3_\nr(X,\Q_{l}/\Z_{l}(2))$  est fini si et seulement si 
l'application $CH^2(X)\otimes \Q_{l} /\Z_{l} \to H^4_{\et}(X,\Q_{l}/\Z_{l}(2))$ a un noyau
fini.
\end{itemize}
\end{prop}
 
En 1968, Bass a demand\'e 
si  le  groupe  $K_{0}(X)$
d'une vari\'et\'e $X$  lisse sur un corps fini est de type fini.
Si c'est le cas, en utilisant le th\'eor\`eme de Riemann-Roch sans d\'enominateurs
\cite[Formule (15)]{grothendieckchern}
cela implique que  le  groupe $CH^2(X)$  est de type fini, 
et donc que le groupe $H^3_\nr(X,\Q_{l}/\Z_{l}(2))$ est de cotype fini.

\subsection{Rappels sur les conjectures de Tate et Beilinson}\label{state0}

Soit $\F$ un corps fini de caract\'eristique $p$.  Fixons une cl\^oture alg\'ebrique
$\overline{\F}$ de $\F$. Soit $G=\Gal(\overline{\F}/\F)$. 

Soit $X$ une $\F$-vari\'et\'e projective
lisse et soit $i$ un
 entier $\geq 0$, enfin soit $l$ un nombre premier ($l=p$ est permis). La conjecture de Tate
``cohomologique" pour $(X,i,l)$ est:

\begin{conj} \label{ctate}  
L'application classe de cycle g\'eom\'etrique 
\[CH^{i}(X) \otimes \Q_{l} \to H^{2i}_{\et}(\bar X,\Q_{l}(i))^G\]
est surjective.
\end{conj}

En utilisant les conjectures de Weil, on voit que cet \'enonc\'e est \'equivalent au suivant:

\begin{conj} \label{ctate1}  
L'application cycle \eqref{eqtate}
$$CH^{i}(X) \otimes \Z_{l} \to H^{2i}_{\et}(X,\Z_{l}(i))$$
a un conoyau fini.
\end{conj}
 
La forme ``forte" de la conjecture de Tate est:

\begin{conj}\label{ctateforte} L'ordre du p\^ole de $\zeta(X,s)$ en $s=i$ est \'egal au rang du groupe des cycles de codimension $i$ sur $X$, modulo l'\'equivalence num\'erique.
\end{conj}

Cette conjecture entra\^\i ne la conjecture \ref{ctate}, \emph{cf}. \S \ref{s3.5}.

Soit $A^{i}(X)$ le groupe des cycles de codimension $i$  sur $X$ modulo l'\'equiva\-lence num\'erique. On sait que c'est un groupe ab\'elien libre de type fini.  La conjecture de Beilinson est :

\begin{conj}\label{cbeil} L'application  $CH^{i}(X)\otimes \Q \to A^{i}(X) \otimes \Q$
est un isomorphisme.
 \end{conj}

\subsection{Les classes $B(\F)$, $B_\num(\F)$ et $B_\Tate(\F)$} Soit $V(\F)$ l'ensemble
des classes d'isomorphismes de vari\'et\'es projectives lisses sur $\F$. Rappelons de
\cite[p. 978, d\'ef. 1]{kahnENS}:

\begin{defi} \label{d3.1}
a) On note $B(\F)$ le sous-ensemble de $V(\F)$ form\'e des vari\'et\'es dont le motif de Chow
(\`a coefficients rationnels) est dans la sous-cat\'egorie \'epaisse rigide engendr\'ee par les
motifs d'Artin et les motifs de vari\'et\'es ab\'eliennes (ou de courbes, c'est la m\^eme
chose).\\ 
On dit que $X$ est \emph{de type ab\'elien} si $X\in B(\F)$.\\
b) On note $B_\Tate(\F)$ le sous-ensemble de $B(\F)$ form\'e des vari\'et\'es $X$ v\'erifiant
la conjecture \ref{ctateforte} en toute codimension.
 \end{defi}

On montre facilement que la condition $X\in B(\F)$ est \'equivalente \`a la suivante: il existe
une vari\'et\'e ab\'elienne $A$ d\'efinie sur $\F$ telle que le motif de Chow de $X$ devienne
facteur direct de celui de $A$ apr\`es une extension finie (ou sur la cl\^oture alg\'ebrique).

\begin{propr}\label{ex2.1}\
\begin{thlist} 
\item $B(\F)$ est stable par produits directs.
\item Si $f:X\to Y$ est un morphisme surjectif, alors $X\in B(\F)$ $\Rightarrow$ $Y\in B(\F)$
(le motif de $Y$ est facteur direct de celui de $X$).
\item Si $X\in B(\F)$ et $\dim X\le 3$, tout \'eclatement de $X$ de centre lisse est dans
$B(\F)$.
\item Par un r\'esultat c\'el\`ebre de Katsura-Shioda \cite{ks}, les hypersurfaces de Fermat sont dans
$B(\F)$.
\item On sait montrer que de nombreuses vari\'et\'es ab\'eliennes sont dans $B_\Tate(\F)$, par
exemple les  produits de courbes elliptiques (Spie\ss\ \cite{spiess}). Pour d'au\-tres exemples, voir \cite[Ex. 1
c)]{kahnENS}.
\item Si  $X\in B(\F)$ et $\dim X\le 3$, alors $X\in B_\Tate(\F)$.   Cela r\'esulte de la proposition \ref{diversesimplications} b) (ii) ci-dessous.
\item  Notons $B_\num(\F)$ le sous-ensemble de $V(\F)$ d\'efini comme $B(\F)$, mais en rempla\c
cant \'equivalence rationnelle par \'equivalence num\'erique:
 il contient \'evidemment
$B(\F)$ et lui est \'egal sous la conjecture \ref{cbeil}. D'apr\`es Milne \cite[rem.
2.7]{milne}, la conjecture
\ref{ctateforte}  implique que toute
$\F$-vari\'et\'e projective  lisse
$X$ a sa classe dans $B_\num(\F)$ (ce r\'esultat utilise les
conjectures de Weil et un th\'eor\`eme de Honda). Ainsi, 
les conjectures \ref{ctateforte} et
\ref{cbeil} entra\^\i nent que \emph{toute vari\'et\'e projective lisse est dans $B_\Tate(\F)$}.

\end{thlist}
\end{propr}

\subsection{Relations entre les conjectures \ref{ctate}, \ref{ctateforte} et \ref{cbeil}}\label{s3.5}

\begin{prop}\label{diversesimplications} 
Soit $(X,i)$ comme ci-dessus, avec $\dim X = d$, et soit $l$ un nombre premier. 
Alors: \\
a) La conjecture \ref{ctateforte} pour $(X,i)$ implique la conjecture  \ref{ctate} pour
$(X,i,l)$. 
 Elle implique aussi, pour tout $l$:
\begin{enumerate}
\item La condition $S^i(X,l)$: la composition
\[H^{2i}(\ovX,\Q_l(i))^G\to H^{2i}(\ovX,\Q_l(i))\to H^{2i}(\ovX,\Q_l(i))_G\]
est un isomorphisme.
\item \'Equivalence homologique $l$-adique et \'equivalence num\'erique co\"\i ncident sur $CH^i(X)\otimes \Q_l$.
\end{enumerate}
b) R\'eciproquement :
 \begin{thlist}
\item 
 La conjecture \ref{ctate} pour $(X,1,l)$ implique la conjecture  \ref{ctateforte} pour $(X,1)$
et $(X,d-1)$ (et donc la conjecture \ref{ctate} pour $(X,d-1,l)$).
\item  Si $X\in B(\F)$, 
 la condition $S^i(X,l)$ est vraie pour tout $(i,l)$. De plus, la conjecture  \ref{ctate} pour $(X,i,l)$ et $(X,d-i,l)$ implique la
conjecture \ref{ctateforte} pour $(X,i)$ et $(X,d-i)$. 
Ces conjectures sont vraies pour $i=1$.

\end{thlist}
c) Dans les cas de b), la conjecture \ref{ctate} ne d\'epend pas du choix de $l$.\\
d) Si $X\in B(\F)$, alors $X\in B_\Tate(\F)$ $\iff$ il existe $l$ tel que $X$ v\'erifie la
conjecture \ref{ctate} en toute codimension.\\ e) Si $X\in B(\F)$, alors les conjectures
\ref{ctate} (pour tout $l$) et \ref{ctateforte} sont vraies pour
$(X,1)$ et $(X,d-1)$.
\end{prop}

\begin{proof}  Cela r\'esulte des propositions 8.2 et 8.4 de Milne
\cite{milneamer},  r\'esum\'ees dans \cite[th. 2.9]{tate}: les arguments de \cite[\S
8]{milneamer} sont explicitement r\'edig\'es sans distinction entre $l\ne p$ et
$l=p$\footnote{et valent en $l=p$ quelle que soit la d\'efinition de la classe de cycle
$p$-adique, puisqu'ils n'en utilisent que les propri\'et\'es formelles.}.  

 La condition $S^i(X,l)$ est \'equivalente \`a la suivante: 
$1$ n'est pas une racine multiple du
 polyn\^ome minimal de l'endomorphisme de Frobenius
sur $H^{2i}(\ovX,\Q_l(i))$.

D'apr\`es la partie (b) de \cite[th. 2.9]{tate}, la conjecture \ref{ctateforte} pour $(X,i)$
est \'equivalente \`a  la conjecture \ref{ctate} pour $(X,i,l)$, jointe \`a la condition
1 de a). D'autre part, la partie (c) de \cite[th. 2.9]{tate}
dit que la conjecture \ref{ctateforte} pour $(X,i)$ est \'equivalente \`a la conjecture \ref{ctate}
pour $(X,i,l)$ et $(X,d-i,l)$ jointe \`a la condition 2 de a). Cela d\'emontre a) et la premi\`ere affirmation de b) (i), puisque la 
condition 1 de a) est connue en codimension~$1$.  En appliquant ceci pour $i=1$ et en tenant compte de la premi\`ere
affirmation de (i), on obtient la seconde.

Pour $l\ne p$, la premi\`ere assertion de (ii) r\'esulte du fait que l'action de Frobenius sur
$H^*(\ovX,\Q_l)$ est semi-simple pour tout $X\in B(\F)$. Pour le voir, on se ram\`ene
imm\'ediatement au cas o\`u $X$ est une vari\'et\'e ab\'elienne $A$, et alors cela r\'esulte
de  \cite[lemme 1.9]{kahnENS}. Pour $l=p$, Milne affirme que $S^i(X,p)$ est encore vrai; Illusie
nous en a fourni une justification. La seconde assertion de (ii) en d\'ecoule alors, comme ci-dessus.

La troisi\`eme assertion de (ii) (cas $i=1$) d\'ecoule du th\'eor\`eme de Tate sur les endomorphismes de vari\'et\'es ab\'eliennes \cite{tatefinite}.
	
c) et d) d\'ecoulent alors de b), puisque la conjecture \ref{ctateforte} ne fait pas
intervenir $l$. \end{proof}

 \begin{rema}\label{tateviaBrauer} 
Par la suite exacte de Kummer, resp. d'Artin-Schreier-Witt, la
conjecture
\ref{ctate1} et donc la conjecture \ref{ctate} 
\'equivaut pour $i=1$  \`a la surjectivit\'e de l'application
$$CH^{1}(X) \otimes \Z_{l} \to H^{2}_{\et}(X,\Z_{l}(1))$$
ou encore \`a la finitude du sous-groupe de torsion $l$-primaire
du groupe de Brauer $\Br(X)$. En utilisant les parties a) et b) (i) de la proposition
\ref{diversesimplications}, on retrouve le fait connu que cette finitude ne d\'epend pas de $l$.
\end{rema}

\begin{theo}\label{annens1}
Si $X\in B_\Tate(\F)$, la conjecture \ref{ctate} implique la conjecture \ref{cbeil}.
\end{theo}

\begin{proof} C'est le contenu de \cite[th. 1.10]{kahnENS}.
\end{proof}

\subsection{Cons\'equences des conjectures de Tate et Beilinson}\label{state}

Dans \cite{kahnENS}, le second auteur a \'etabli un certain nombre de
propri\'et\'es de la cohomologie motivique de ces vari\'et\'es.
Pour les cycles de codimension~2, on a en particulier le
th\'eor\`eme suivant, dont la preuve utilise le th\'eor\`eme \ref{annens1}.
 
\begin{theo}[\cite{kahnENS}]\label{copiekahn}    Soit $\F$ un corps fini et $X$ une
$\F$-vari\'et\'e projective, lisse, g\'eom\'etriquement connexe, dans la classe  $ B_\Tate(\F)$.
Alors :
\begin{itemize}
\item[(a)] Le groupe $CH^2(X)$ est un groupe de type fini.
\item[(b)]  Le groupe $H^4_\et(X,\Z(2))$ est un groupe de type fini.
\item[(c)] Le groupe $H^3_\nr(X,\Q/\Z(2))$  est fini.
\end{itemize}
\end{theo}

Les \'enonc\'es (a) et  (c) r\'esultent de l'\'enonc\'e (b), cas particulier de  \cite[Cor.
3.10]{kahnENS},  et de la suite exacte \eqref{eqLK}.

Ainsi,  en tenant compte du r\'esultat de Milne  \ref{ex2.1}  (vii), la combinaison des
 conjectures \ref{ctateforte}  (conjecture de Tate sur les p\^oles de la fonction z\^eta)
 et \ref{cbeil} 
  (conjecture  de Beilinson sur la co\"{\i}ncidence de l'\'equivalence de Chow et de
l'\'equivalence
  num\'erique \`a coefficients $\Q$)
implique que
 \emph{$H^3_\nr(X,\Q/\Z(2))$  est fini pour toute vari\'et\'e projective lisse sur un corps
fini}.

Le corollaire suivant permet parfois de montrer que $H^3_\nr(X,\Q/\Z(2))$ est fini
sans
savoir que $X$ est dans  $ B_\Tate(\F)$. Il r\'esulte imm\'ediatement de la combinaison du th\'eor\`eme
\ref{copiekahn}, de la proposition \ref{p4.1} et de l'exemple \ref{ex2.1} (vi).

\begin{cor}\label{correspBTate}
Soit $X/\F$ une vari\'et\'e projective lisse. On suppose qu'il existe une
$\F$-vari\'et\'e projective lisse $Y$ et une correspondance de Chow
 $\gamma\in CH_{\dim Y}(Y\times 
X)\otimes \Q$ telles que:
\begin{thlist}
\item $Y\in B_\Tate(\F)$.
\item $\gamma_*:CH_0(Y_\Omega)\otimes \Q\to CH_0(X_\Omega)\otimes \Q$ est surjectif.
\end{thlist}
Alors $H^3_\nr(X,\Q/\Z(2))$ est fini.\qed
\end{cor}

Par exemple, il suffit de trouver $(Y,\gamma)$ avec $\dim Y\le 3$ et $Y\in B(\F)$.

 \subsection{Conjecture de Tate  sur les diviseurs et sur les $1$-cycles}

 Soit $X$
une
$\F$-vari\'et\'e projective lisse de dimension $d$, et soit $l$ un nombre premier. 
Comme on le rappellera au \S \ref{rappelscyclesglobalpositif}, on s'int\'eresse particuli\`erement
au conoyau de l'application cycle
$$CH^{d-1}(X) \otimes \Z_{l} \to H^{2d-2}_{\et}(X,\Z_{l}(d-1)),$$ 
dont on se demande s'il est nul.

Supposons que
la conjecture
\ref{ctate} soit vraie pour $(X,1,l)$. La proposition
\ref{diversesimplications} b) (i) implique alors qu'elle est vraie pour $(X,1,l')$ et
$(X,d-1,l')$ pour tout nombre premier $l'$; de mani\`ere \'equivalente (conjecture
\ref{ctate1}), pour $i=1$ et $i=d-1$  le conoyau de \eqref{eqtate}  est fini. 

Nous allons raffiner
ce r\'esultat. Pour cela, nous avons besoin de la condition $S^i(X,l)$  de la proposition \ref{diversesimplications} a) 1; on y a vu (b) (ii)) qu'elle est vraie pour tout
tout $l$ 
si $X\in B(\F)$.

\begin{lem}\label{lbrf} Soit $i\in [0,d]$. Il existe un entier $N=N(X,i)>0$ ayant la
propri\'et\'e suivante: pour
$l>N$, si la condition $S^i(X,l)$ est vraie, alors l'accouplement
\begin{multline*}
H^{2i}_\et(X,\Z_l(i))\times H^{2(d-i)}_\et(X,\Z_l(d-i))\\
\to H^{2d}_\et(X,\Z_l(d))\to
H^{2d}_\et(\ovX,\Z_l(d))\simeq \Z_l
\end{multline*}
est parfait (en particulier, ses deux termes sont sans torsion).
\end{lem}

\begin{proof}

On dispose de l'action du Frobenius g\'eom\'etrique $F$  sur chaque  $ \Z_{l} $-modules $H^{i}(\ovX,\Z_{l}(j))$
et sur le $\Q_{l}$-vectoriel correspondant  $H^{i}(\ovX,\Q_{l}(j))$.

D'apr\`es Deligne \cite{deligne}, le polyn\^ome caract\'eristique inverse $\det(1-F t)$  
pour cette action est un polyn\^{o}me \`a coefficients dans $\Q[t]$
 ind\'ependant de $l$,
et qui  pour $i\neq 2j$ ne s'annule pas en $t=1$.

D'apr\`es Gabber \cite{gabber}, pour presque tout
nombre premier $l$,
les $\Z_{l}$-modules  $H^{i}(\ovX,\Z_{l}(j))$ sont sans torsion. 
Il en r\'esulte que  pour presque tout premier $l$, 
pour $i \neq 2j$, $H^{i}(\ovX,\Z_{l}(j))$ est sans torsion et
l'endomorphisme
$F-1$ est un automorphisme de ce module, et donc
  $H^1(G,H^{i}(\ovX,\Z_{l}(j)))=0$.

De la suite spectrale de Hochschild-Serre, on d\'eduit pour tout $i$
des suites exactes courtes
$$0 \to H^1(G,H^{2i-1}_{\et}(\ovX,\Z_{l}(i)))) \to H^{2i}_{\et}(X,\Z_{l}(i)) \to H^{2i}_{\et}(\ovX,\Z_{l}(i)))^G \to 0.$$

On obtient donc que pour presque tout $l$,
on a un isomorphisme de $\Z_{l}$-modules sans torsion
$$H^{2i}_{\et}(X,\Z_{l}(i)) \iso H^{2i}_{\et}(\ovX,\Z_{l}(i)))^G.$$

Soit $P(t)=\det(1-F t)$ le polyn\^{o}me caract\'eristique inverse de $F$
sur $H^{2i}_{\et}(\ovX,\Q_{l}(i))$, qui est comme on a dit dans 
$\Q[t]$ et ind\'ependant de $l$. \'Ecrivons $P(t)=(t-1)^m.Q(t)$ dans $\Q[t]$,
avec $Q(1) \neq 0$. On a une \'egalit\'e de B\'ezout
\begin{equation}\label{eqbezout}
a(t)(t-1)+b(t)Q(t)=1 \in \Q[t].
\end{equation}

Sous l'hypoth\`ese $S^i(X,l)$, l'endomorphisme $(F-1).Q(F)$ s'annule sur  $V_{l}:=H^{2i}_{\et}(\ovX,\Q_{l}(i))$,
et 
$$V_{l} = V_{l}^{F=1} \oplus V_{l}^{Q(F)=0}.$$
Notant $M_{l} = H^{2i}_{\et}(\ovX,\Z_{l}(i))$, utilisant le fait que
\eqref{eqbezout} est \`a coefficients $l$-entiers pour presque tout $l$, on
en conclut que pour presque tout $l$ satisfaisant $S^i(X,l)$, on a une d\'ecomposition
$$M_{l} = M_{l}^{F=1} \oplus M_{l}^{Q(F)=0},$$
avec $M_{l}$ sans torsion, et une \'egalit\'e
$$a(F)(F-1)+b(F)Q(F)=1  \in \Z_{l}[F].$$
Ceci implique que pour presque tout premier $l$, si $S^i(X,l)$ est v\'erifi\'e, l'application compos\'ee
$$M_{l}^G=M_{l}^{F=1} \to M_{l} \to  M_{l}/(F-1)M_{l} = ({M_{l}})_{G}$$
est un isomorphisme de $\Z_{l}$-modules sans torsion.

Il est connu
que pour tout premier $l$ diff\'erent de la caract\'eristique,
 l'accouplement de Poincar\'e
\[H^{2i}_\et(\ovX,\Z_l(i))\times H^{2(d-i)}_\et(\ovX,\Z_l(d-i))\to\Z_l\]
est parfait modulo torsion.  (Voir \cite{zarhin} pour une d\'emonstration r\'ecente.)

Ce qui pr\'ec\`ede implique que pour presque tout  nombre premier $l$ satisfaisant $S(X,i,l)$,
l'accouplement induit 
\[H^{2i}_\et(\ovX,\Z_l(i))_G\times H^{2(d-i)}_\et(\ovX,\Z_l(d-i))^G\to\Z_l\]
est une dualit\'e parfaite de $\Z_{l}$-modules sans torsion, et
qui s'identifie \`a l'accouplement
$$H^{2i}_{\et}(X,\Z_{l}(i))  \times H^{2(d-i)}_{\et}(X,\Z_{l}(d-i))) \to \Z_{l}$$
par les fl\`eches \'evidentes. Pour presque tout $l$ satisfaisant $S(X,i,l)$, l'accouplement
$$H^{2i}_{\et}(X,\Z_{l}(i))  \times H^{2(d-i)}_{\et}(X,\Z_{l}(d-i))) \to \Z_{l}$$
est donc parfait.
\end{proof}

\begin{prop}\label{brauerfini}
Si la Conjecture \ref{ctate} est vraie pour $(X,1)$ 
(pour un nombre premier $l_0$), alors l'application
cycle
\[CH^{d-1}(X)\otimes \Z_l\to H^{2d-2}(X,\Z_l(d-1))\]
a son conoyau fini, et nul pour presque tout $l$.
\\
En particulier, cette conclusion vaut pour tout $X\in B(\F)$.
\end{prop}

\begin{proof} 
La finitude du conoyau de l'application pour chaque premier $l$
est un cas particulier de la proposition \ref{diversesimplications} b  (i).

Notons $B^i(X,\Z_l)$ l'image de l'application cycle \eqref{eqtate}: 
comme l'\'equivalence homologique sur les cycles implique
l'\'equivalence num\'erique,
 on a une surjection \[\phi^i_l:B^i(X,\Z_l)\Surj A^i(X)\otimes \Z_l.\] 

D'apr\`es la proposition \ref{diversesimplications}, la conjecture \ref{ctate} pour $(X,1,l_0)$ implique la conjecture \ref{ctateforte} pour $(X,1)$ et $(X,d-1)$, donc a fortiori:
\begin{itemize}
\item La conjecture \ref{ctate1} est vraie pour $(X,1,l)$ et $(X,d-1,l)$ quel que soit $l$.
\item La condition $S^1(X,l)$ est vraie quel que soit $l$.
\item Le noyau de $\phi^{d-1}_l$ est de torsion quel que soit $l$.
\end{itemize}

Si $l>N(X,1)$ o\`u $N(X,1)$ est l'entier du lemme \ref{lbrf}, $\phi^{d-1}_l$ est bijectif puisque $H^{2d-2}(X,\Z_l(d-1))$ est sans torsion.

Consid\'erons l'accouplement de $\Z$-modules libres de type fini
\[A^1(X)\times A^{d-1}(X)\to \Z.\]

Par d\'efinition de l'\'equivalence num\'erique, il est non d\'eg\'en\'er\'e. Soit $D$ son
discriminant: pour $l>D$, cet accouplement devient parfait apr\`es tensorisation par $\Z_l$.
Si $l>\sup(D,N(X,1))$, l'accouplement du lemme  \ref{lbrf} est \'egalement parfait, et compatible avec l'accouplement
d'intersection sur les groupes de Chow.

Le conoyau de $B^{d-1}(X,\Z_l) \hookrightarrow  H^{2d-2}(X,\Z_l(d-1))$ est un module de torsion.
Montrons qu'il est nul.
Soit $x\in H^{2d-2}(X,\Z_l(d-1))$ tel que $lx=\cl(y)$, pour $y\in B^{d-1}(X,\Z_l)$. Pour
tout $y'\in B^1(X,\Z_l)$, le produit d'intersection $\langle y,y'\rangle$ est divisible par $l$.
Par cons\'equent, $y$ est divisible par $l$ dans $A^{d-1}(X)\otimes \Z_l$, 
 donc
dans $B^{d-1}(X,\Z_l)$. Ceci conclut la d\'emonstration.

La derni\`ere affirmation r\'esulte de la proposition \ref{diversesimplications} b) (ii).
\end{proof}

\begin{cor}\label{c3.22} Si $X$ de dimension $3$ v\'erifie la conjecture de Tate pour les
diviseurs,  le groupe $H^3_\nr(X,\Q_l/\Z_l(2))$ est divisible pour $l >>0$.
\end{cor}

\begin{proof} Cela r\'esulte de la proposition \ref{brauerfini} et du
th\'eor\`eme \ref{Hauptsatz}.
\end{proof}

\begin{prop} \label{courbedomine}
Soit  $\F$ un corps fini.
Soit $X$ une $\F$-vari\'et\'e projective, lisse, g\'e\-o\-m\'e\-tri\-que\-ment int\`egre
de dimension $d$. 
S'il existe une extension finie $L/\F$, 
un  domaine universel $\Omega$  contenant $L$, 
une courbe $Y$ projective et lisse sur $L$ et un $L$-morphisme $Y
\to X\times_{\F}{L}$  qui induit une surjection
des groupes de Chow de z\'ero-cycles sur le corps $\Omega$,  alors pour tout $l$ premier
distinct de la caract\'eristique,  le groupe 
\[\Coker \left(CH^{d-1}(X) \otimes \Z_{l} \to H^{2d-2}_{\et}(X,\Z_{l}(d-1))\right)\]
est   fini, et il est nul pour presque tout $l$.
 \end{prop}
\begin{proof}   
Comme le  grou\-pe de Brauer d'une courbe projective et lisse
  sur un corps fini est nul, 
 un argument  de
  correspondances enti\`ere\-ment analogue \`a celui donn\'e
  dans la proposition \ref{surfacedomine}
  montre que le groupe de Brauer $\Br(X)$  
 est annul\'e par un entier $N>0$, donc est, \`a la torsion $p$-primaire pr\`es, fini. 
 D'apr\`es la remarque la remarque \ref{tateviaBrauer}, la conjecture
  \ref{ctate} vaut pour $(X,1)$.
 On conclut alors avec la proposition \ref{brauerfini}.
\end{proof}

 \section{Divisibilit\'e, voire nullit\'e, de $H^3$ non ramifi\'e: exemples} \label{taillefiniH3}

\subsection{Solides sur  la cl\^oture alg\'ebrique d'un corps fini}
 
 \begin{lem}\label{lemmegaloisien}
 Soient $G$ un groupe profini  et $M$ un $\Z_{l}$-module   de type fini
muni d'une action continue de $G$. Soit $M^{(1)} \subset M$ le $\Z_{l}$-module form\'e des
\'el\'ements dont le stabilisateur est un sous-groupe ouvert de $G$. Le quotient $M/M^{(1)}$
est sans torsion.
 \end{lem}

\begin{proof} a) Soit $M_\tors$ le sous-module de torsion de $M$, et soit $$\bar M=M/M_\tors.$$
Comme $M_\tors$ est fini, on a
\[\colim_{U\subset G} H^0(U,M_\tors)=M_\tors,\qquad \colim_{U\subset G} H^1(U,M_\tors)=0\]
o\`u $U$ d\'ecrit les sous-groupes ouverts de $G$. On en d\'eduit une suite exacte
\[0\to M_\tors \to M^{(1)}\to \bar M^{(1)}\to 0.\]

Ainsi, $M/M^{(1)}\iso \bar M/\bar M^{(1)}$ et on peut supposer $M$ sans torsion.

b) Supposons $M$ sans torsion. Comme $M$ est un
$\Z_l$-module noeth\'erien, on a $M^{(1)}=M^U$ pour un sous-groupe ouvert $U\subset G$ assez
petit.  Soit $m\in M$ tel que $lm\in M^{(1)}$, c'est-\`a-dire $(u-1)(lm)=0$ pour tout $u\in U$.
Alors $(u-1)m=0$ pour tout $u\in U$, donc $m\in M^{(1)}$.
\end{proof}

 \begin{prop}\label{dim3algclos} 
Soit $V$ une vari\'et\'e projective, lisse, connexe
de dimension $3$ sur la cl\^oture alg\'ebrique $\overline{\F}$ d'un corps fini $\F$.
Sous la conjecture de Tate pour les classes de diviseurs sur les surfaces
sur  un corps fini,
pour tout nombre premier $l$ diff\'erent de la caract\'eristique
de $\F$,
le groupe $$H^3_\nr(V,\Q_{l}/\Z_{l}(2))$$ est un groupe divisible.
\end{prop}

\begin{proof}
Il existe un corps fini $\F$ et une $\F$-vari\'et\'e projective, lisse, g\'eom\'etriquement int\`egre
$X$ telle que $V=X \times_{F} \overline{\F}$. Soit $G={\Gal}(\overline{\F}/\F)$.
D'apr\`es un th\'eor\`eme de Schoen (\cite{schoen}, voir aussi \cite{ctsz}), 
pour $V$ de dimension $d$, la conjecture de Tate
pour les diviseurs sur les surfaces  implique que l'image de la classe de cycle
$$CH^{d-1}(V) \otimes \Z_{l} \to  H^{2d-2}_{\et}(V,\Z_{l}(2))$$
 est pr\'ecis\'ement le 
$\Z_{l}$-sous-module des \'el\'ements dont le stabilisateur est ouvert dans $G$.
Il r\'esulte alors du lemme \ref{lemmegaloisien} que le conoyau de l'application
cycle ci-dessus est sans torsion. Pour $d=3$, le th\'eor\`eme \ref{Hauptsatz} implique alors que
le groupe $H^3_\nr(V,\Q_{l}/\Z_{l}(2))$ est divisible.
\end{proof}

\begin{rema}
On notera  que sur le corps des complexes on dispose de vari\'et\'es
$V$ projectives et lisses de dimension 3 avec pour $l$ premier convenable  $H^3_\nr(V,\Q_{l}/\Z_{l}(2))$
non divisible, c'est-\`a-dire (\cite[Thm. 3.7]{ctv}  ou Th\'eor\`eme  \ref{Hauptsatz} ci-dessus)
que   la torsion du groupe 
 $$M = \Coker \left(CH^2(V) \otimes \Z_{l} \to H^4_{\et}(V,\Z_{l}(2))\right)$$
est non nulle. De tels exemples,    avec $M$ fini, ont \'et\'e donn\'es par
Koll\'ar (voir \cite[\S 5.3]{ctv}).
\end{rema}

\subsection{Solides sur un corps fini  fibr\'es en coniques}
 
 \begin{theo}[Parimala et Suresh \protect{\cite{parimalasuresh}}] \label{H3parimalasuresh}
  \footnote{Ce r\'esultat vient d'\^etre \'etendu par A. Pirutka  \cite{pirutka2} aux
fibrations en vari\'et\'es de Severi-Brauer d'indice premier au-dessus d'une surface.}
Soit $f : X \to S$ un morphisme surjectif de vari\'et\'es projectives lisses sur un corps fini
 $\F$ de caract\'e\-ris\-tique $p$ diff\'erente de $2$. Supposons que $S$ soit une surface et que la
fibre g\'en\'erique de $f$
 soit une conique (lisse et g\'eom\'etriquement int\`egre sur $\F(S)$). Alors pour tout nombre
premier $l
\neq p$, on a 
$$H^3_\nr(X,\Q_{l}/\Z_{l}(2))=0.$$

\end{theo}

\begin{rema}
Remarquons que  $H^3_\nr(X,\Q/\Z(2))$ est a priori annul\'e par~$2$, en particulier fini
d'apr\`es le th\'eor\`eme 
\ref{Hauptsatz}. En effet, soient
$K=\F(S)$ et $F=\F(X)$. Alors $F=K(C)$ o\`u $C$ est une conique. Si $c$ est un point de degr\'e
$2$ de
$C$ et $K'=K(c)$, alors $F'=K'F$ est $K'$-isomorphe \`a $K'(t)$. Donc
\[H^3_\nr(K'/\F,\Q/\Z(2))\iso H^3_\nr(F'/\F,\Q/\Z(2))\]
(l'isomorphisme est classique hors de la $p$-torsion; pour celle-ci, cf. \cite[ex.
7.4 (3) et th. 8.6.1]{ctkh}), et le groupe de gauche est nul par la proposition
\ref{ctssgros}. On conclut par l'argument habituel de transfert.
\end{rema}

\begin{rema}
Dans le th\'eor\`eme \ref{H3parimalasuresh}, le cas d'une fibration lisse
 est simple \`a
expliquer:

  En gardant les notations ci-dessus, l'application naturelle 
\[H^3(K,\Q_{l}/\Z_{l}(2)) \to H^3(F,\Q_{l}/\Z_{l}(2))\]
  induit une surjection (valable en g\'en\'eral pour le corps des fonctions d'une conique
\cite{suslin})
 $$ H^3(K,\Q_{l}/\Z_{l}(2)) \to H^3_\nr(F/K,\Q_{l}/\Z_{l}(2)) \to 0.$$

Soit $\beta\in H^3(K,\Q_{l}/\Z_{l}(2))$. Soient $s\in S$ un point de codimension $1$ et
$x=f^{-1}(s)\in X$. Sous l'hypoth\`ese de lissit\'e de $f:X \to S$, $x$ est un point de
codimension $1$ et l'extension $\sO_{S,s}\subset \sO_{X,x}$ est non ramifi\'ee; d'o\`u
\[\partial_x(f^*\beta)=f^*\partial_s(\beta)\in \Br(K(x)).\]

Toujours sous l'hypoth\`ese de lissit\'e, la fibration en coniques
  est associ\'ee \`a une alg\`ebre d'Azumaya $A$ sur $S$ (de dimension 4). D'apr\`es un
th\'eor\`eme classique de Witt, on a
\[\Ker\left(\Br(\F(s))\to \Br(K(x))\right) =\langle A_s\rangle\]
o\`u $A_s$ d\'esigne la fibre de $A$ en $s$. Mais $A_s=0$
  car elle provient d'un \'el\'ement du groupe de Brauer d'une courbe projective lisse
  sur un corps fini (\`a savoir, la normalisation de l'adh\'erence de $x$ dans $X$).

Ainsi, si $\beta$  devient non ramifi\'ee
dans $H^3(F,\Q_{l}/\Z_{l}(2))$, alors $\partial_s(\beta)=0$
pour tout  $s\in S^{(1)}$. Finalement le morphisme $H^3_\nr(K/\F,\Q_{l}/\Z_{l}(2)))\to
H^3_\nr(F/\F,\Q_{l}/\Z_{l}(2)))$ est surjectif, et on conclut par la proposition
\ref{ctssgros}.

\end{rema}

 \section{Questions et conjectures}\label{exemplesquestions}

Soient $\F$ un corps fini, $p =\car(\F)$ et  $\overline{\F}$ une cl\^oture alg\'ebrique.
Soit $l$ premier, $l\neq p$. 
En  grande dimension, une adaptation des exemples d'Atiyah-Hirzebruch 
(voir \cite{ctsz}) donne  des exemples de $V/{\overline \F}$ projectives et lisses 
et de classes dans $H^4_{\et}(V,\Z_{l}(2))_{\tors}$ qui ne sont pas
dans l'image de l'application cycle 
$$CH^2(V) \otimes \Z_{l} \to H^4_{\et}(V,\Z_{l}(2)).$$
On en d\'eduit imm\'ediatement 
des exemples de vari\'et\'es $X/\F$
projectives et lisses 
et de classes dans $H^4_{\et}(X,\Z_{l}(2))_{\tors}$ qui ne sont pas
dans l'image de l'application cycle 
$$CH^2(X) \otimes \Z_{l} \to H^4_{\et}(X,\Z_{l}(2)).$$

En utilisant le th\'eor\`eme \ref{Hauptsatz}, ceci donne
des exemples de $V/{\overline \F}$  pour lesquelles le groupe $H^3_\nr(V,\Q_{l}/\Z_{l}(2)) $ est non divisible,
et donc non nul, 
et des exemples de $X/\F$ avec
$H^3_\nr(X,\Q_{l}/\Z_{l}(2))$ non divisible, et donc non nul.

A. Pirutka \cite{pirutka}  a par ailleurs 
donn\'e des exemples de vari\'et\'es $X/\F$
 projectives et lisses  g\'eom\'etriquement rationnelles
 de dimension~5 
avec  $$H^3_\nr(\ovX,\Q_{l}/\Z_{l}(2))=0$$ et $H^3_\nr(X,\Q_{l}/\Z_{l}(2))$ fini non nul.

Ces exemples et le paragraphe \ref{taillefiniH3}
 laissent les conjectures et questions suivantes ouvertes.   La premi\`ere  conjecture
  r\'esulte   de la combinaison de la conjecture  \ref{ctateforte} de Tate sur les p\^oles de la fonction z\^eta et de la conjecture \ref{cbeil} de Beilinson, nous  la  r\'ep\'etons    pour m\'emoire (cf. \S \ref{state}):

    \begin{conj}\label{questionfinie}  Le groupe $H^3_\nr(X,\Q/\Z(2))$ est fini
pour toute $X/\F$ projective et lisse.
    \end{conj}

 \begin{qn}\label{questiongeom}
  Existe-t-il $V/{\overline \F}$ projective et lisse,   avec  
$H^3_\nr(V,\Q_{l}/\Z_{l}(2))$ infini ?
\end{qn}

Voir  \cite[prop. 5.5]{kahncycle}  pour une liste de conditions \'equivalentes \`a la finitude de ce groupe
et
 \cite[th. 5.6]{kahncycle} pour un exemple non trivial o\`u cette finitude est \'etablie.

   \begin{qn}\label{questiondim3geom} Pour 
 $V/{\overline \F}$ projective et lisse,  de dimension 3, a-t-on
$$H^3_\nr(V,\Q_{l}/\Z_{l}(2))= 0  \hskip2mm  ?  $$ 
\end{qn}

Ceci serait bien s\^ur le cas si l'on avait une r\'eponse affirmative \`a la question suivante.

 \begin{qn}\label{questiondim3fini}   Pour une vari\'et\'e $X/\F$ projective et lisse de dimension~3, a-t-on
$H^3_\nr(X,\Q_{l}/\Z_{l}(2))=0$  ? 
\end{qn}

\begin{rema}
Soit $X/\F$ projective, lisse, g\'eom\'etriquement connexe de dimension $d$.
On peut s'int\'eresser plus g\'en\'eralement aux groupes 
$H^{i+1}_{\nr}(X,\Q/\Z(i))$, $i \geq 0$.

Le quotient $H^{1}_{\nr}(X,\Q/\Z)/H^1(\F,\Q/\Z) $ est  fini (Weil).   La conjecture \ref{ctate} pour $i=1$ est \'equivalente \`a la finitude de $H^{2}_{\nr}(X,\Q/\Z(1))=\Br(X)$ (Tate, cf. \cite{tate}).

Pour $i=0,1$, on a de nombreux exemples  de vari\'et\'es de dimension $d=i+1$
pour lequelles les groupes ci-dessus sont non nuls.  Il semble difficile de construire
un tel exemple avec $i=2$ et $d=3$ : c'est la question 
  \ref{questiondim3fini}. Elle
est li\'ee \`a des probl\`emes d\'elicats sur le groupe de Griffiths des cycles de codimension 2 (cf. \cite[\S 5]{kahncycle}).

\end{rema}

\begin{rema}
C'est un th\'eor\`eme de Merkurjev et Suslin que l'on a un isomorphisme naturel
$$ H^3_\nr(X,\mu_l^{\otimes 2}) =  H^3_\nr(X,\Q_{l}/\Z_{l}(2))[l].$$
Compte tenu de la suite exacte rappel\'ee au d\'ebut du \S \ref{conjbass}, on voit que 
la nullit\'e  de $H^3_\nr(X,\Q_{l}/\Z_{l}(2))$ 
pour $X$ comme dans la question \ref{questiondim3fini}  est un \'enonc\'e tr\`es fort, car \'equivalent
  \`a la conjonction de :

(a) L'application cycle $CH^2(X)/l \to H^4_{\et}(X,\mu_l^{\otimes 2})$ est injective.

(b) La restriction $H^3_{\et}(X,\mu_l^{\otimes 2}) \to H^3_{\et}(\F(X),\mu_l^{\otimes 2})$
est d'image nulle.
\end{rema}

On dit  qu'une vari\'et\'e g\'eom\'etriquement int\`egre de dimension $d$  est
\emph{g\'eom\'e\-tri\-quement unir\'egl\'ee} si  apr\`es extension finie du corps de base elle est
rationnellement domin\'ee par le produit de la droite projective et d'une vari\'et\'e de
dimension $d-1$. 
  Pour une vari\'et\'e complexe $X$ de dimension $3$   unir\'egl\'ee,  on peut gr\^ace \`a
un  th\'eor\`eme de C.~Voisin montrer  $H^3_\nr(X,\Q/\Z(2))=0$ \cite[Thm. 1.2]{ctv}.
Une vari\'et\'e de dimension $3$ fibr\'ee en coniques sur une surface est unir\'egl\'ee.
Si l'on tient compte du 
th\'eor\`eme \ref{H3parimalasuresh} (Parimala et Suresh) pour de telles vari\'et\'es
d\'efinies sur un corps fini,
on est amen\'e \`a proposer la conjecture suivante. 

\begin{conj}\label{conjunireglee3finie}
Pour $X$ une $\F$-vari\'et\'e projective, lisse, de dimension $3$,
g\'eom\'etriquement unir\'egl\'ee,   le groupe fini   (remarque \ref{remsurfdomine2}) 
$H^3_{\nr}(X,\Q_{l}/\Z_{l}(2)) $ est nul.
\end{conj}

Le cas particulier o\`u  $X$ est g\'eom\'etriquement rationnelle
est d\'ej\`a ouvert.
En voici un autre, qui
 motive en grande partie le pr\'esent article
(voir le    \S \ref{globalpositif} et le \S \ref{global}). 
   
\begin{conj} \label{q4.2} Soit  $f : X \to C$
un morphisme (surjectif \`a fibre g\'en\'erique lisse et g\'eom\'etriquement int\`egre) de
vari\'et\'es projectives, lisses, g\'e\-o\-m\'e\-tri\-que\-ment int\`egres sur un corps fini
 $\F$. On suppose que $C$ est une courbe et que la fibre g\'en\'erique de $f$ une surface
g\'e\-o\-m\'e\-tri\-que\-ment rationnelle. 
 Alors  
 $H^3_\nr(X,\Q_{l}/\Z_{l}(2))=0$.  
\end{conj}

\section{
Descente galoisienne }

 \subsection{Descente galoisienne sur un corps quelconque}
\label{desgal}

L'\'enonc\'e suivant corrige et amplifie la suite ``exacte" (6)
de \cite[p. 397]{kahn}.  Il est utilis\'e pour \'etablir  la proposition \ref{avecbk}.

\begin{prop}\label{p2.1}  Soit $F$ un corps d'exposant caract\'eristique $p$, et soit $A$ une
$F$-alg\`ebre r\'eguli\`ere semi-locale essentiellement de type fini sur $F$,
d'origine g\'eom\'etrique et g\'eom\'etriquement int\`egre. Notons $F_s$ une cl\^oture
s\'eparable de
$F$ et
$A_s=A\otimes_F F_s$ (qui, par hypoth\`ese, est int\`egre). Alors on a un complexe de groupes de
cohomologie
\'etale
\begin{multline*}
H^3(F,\Q/\Z(2))\to \Ker\left(H^3(A,\Q/\Z(2))\to H^3(A_s,\Q/\Z(2))\right)\\
\to H^2(F,K_2(A_s)/K_2(F_s)))\to H^4(F,\Q/\Z(2))\to H^4(A,\Q/\Z(2))
\end{multline*}
qui est exact, sauf peut-\^etre au terme $H^4(F,\Q/\Z(2))$ o\`u son homologie est l'homologie
d'un complexe
\[H^3(A,\Q/\Z(2))\to H^3(A_s,\Q/\Z(2))^G\to H^3(F,K_2(A_s)/K_2(F_s)) \]
o\`u $G=Gal(F_s/F)$.\\
Dans cet \'enonc\'e, $\Q/\Z(2) = \bigoplus_l \Q_l/\Z_l(2)$ o\`u, pour $l\ne p$, $\Q_l/\Z_l(2)$
est un groupe de racines de l'unit\'e tordues, tandis que pour $l=p$ c'est le faisceau
d\'ecal\'e  $\nu_{\infty}(2)[-2]$, cf. \S \ref{outils}.
\end{prop}

\begin{proof} On va suivre la technique de \cite{kahn}, en rempla\c cant les complexes
$\Gamma(X,2)$ de Lichtenbaum par les complexes $\Z_X(2)$ de Bloch, consid\'er\'es pour la
topologie \'etale. Cette op\'eration est
\emph{a priori} un peu d\'esagr\'eable, car ces complexes
 ne sont cohomologiquement born\'es
inf\'e\-rieu\-rement  que conjecturalement (conjecture de Beilinson-Soul\'e), et on travaille avec un corps de
dimension cohomologique \'eventuellement infinie. Que ce probl\`eme soit innocent est
expliqu\'e dans
\cite[\S 2.3]{kahncycle}. 

Soit
$f:\Spec A\to
\Spec F$ le morphisme structural et soit
$\Z(A/F,2)$ ``la" fibre homotopique de $\Z(2)_F\to Rf_*\Z(2)_A$, prise dans la
cat\'egorie d\'eriv\'ee des complexes de $G$-modules. On a une suite spectrale de
Hoch\-schild-Serre
\[E_2^{p,q}=H^p(G,\H^q(\Z(A/F,2)))\Rightarrow \bH^{p+q}(F,\Z(A/F,2))\]

Comme $\Z(A/F,2)$ est un complexe a priori non born\'e, cette suite
spectrale s'obtient en rempla\c cant ce complexe de faisceaux par un
complexe quasi-isomorphe $K$-injectif au sens de Spaltenstein, cf. \cite[Th. 4.5 et Rem.
4.6]{spaltenstein}. Elle converge par l'argument rappel\'e dans \cite[\S 2.3]{kahncycle}, ou
par une variante plus concr\`ete de cet argument donn\'ee ci-dessous.

Consid\'erons le diagramme commutatif de triangles exacts
\[\begin{CD}
\Z(A/F,2)@>>> \Z_F(2)@>>> Rf_*\Z_A(2)@>+1>>\\
@VVV @VVV @VVV\\
\Q(A/F,2)@>>> \Q_F(2)@>>> Rf_*\Q_A(2)@>+1>>\\
@VVV @VVV @VVV\\
\Q/\Z(A/F,2)@>>> (\Q/\Z)_F(2)@>>> Rf_*(\Q/\Z)_A(2)@>+1>>\\
@V{+1}VV @V{+1}VV @V{+1}VV
\end{CD}\]

La derni\`ere ligne et la valeur de $(\Q/\Z)_F(2)$, $(\Q/\Z)_A(2)$ \eqref{LSVGL2} montrent que
$\H^q(\Q/\Z(A/F,2))=0$ pour $q\le 1$. On d\'eduit alors de la colonne de gauche que
$\H^q(\Z(A/F,2))$ est uniquement divisible pour $q\le 2$, ce qui donne la premi\`ere ligne de

\[
\H^q(\Z(A/F,2))=
\begin{cases}
\text{uniquement divisible} &\text{pour $q\le 2$}\\
K_2(A_s)/K_2(F_s) &\text{pour $q=3$}\\
0 &\text{pour $q=4$}\\
H^3(A_s,\Q/\Z(2))) &\text{pour $q=5$}
\end{cases}
\]
cf. \cite[Lemma 3.1]{kahn}. La diff\'erence avec le calcul de \cite{kahn} est le cas $q\le 1$,
o\`u  l'on trouve $0$ (ici ce n'est qu'une conjecture). Le reste du calcul se fait avec les
m\^emes raisonnements que dans loc. cit.

En particulier on a $E_2^{p,q}=0$ pour $p>0,q\le 2$, ce qui assure la convergence de la suite
spectrale.

Comme dans
\cite{kahn2}, \eqref{LSVGL2} donne aussi une suite exacte
\begin{multline}\label{eq3}
0\to \bH^4(F,\Z(A/F,2))\to H^3(F,\Q/\Z(2))\to H^3(A,\Q/\Z(2))\\ 
\to \bH^5(F,\Z(A/F,2))\to H^4(F,\Q/\Z(2))\to H^4(A,\Q/\Z(2))\to\dots 
\end{multline}

 Donc $E_2^{p,q}=0$ pour $q\le 1$ et $q=2, p\ne 0$, d'o\`u un isomorphisme
\begin{equation}\label{eq1a}
H^1(G,K_2(A_s)/K_2(F_s))\iso \bH^4(F,\Z(V/F,2))\iso \Ker \eta
\end{equation}
o\`u $\eta$ est l'homomorphisme $H^3(F,\Q/\Z(2))\to H^3(A,\Q/\Z(2))$ (on a utilis\'e
\eqref{eq3}),  et une suite exacte
\begin{multline}\label{eq2a}
0 \to H^2(G,K_2(A_s)/K_2(F_s))\to \bH^5(F,\Z(A/F,2))\\
\to\bH^5(F_{s},\Z(A_{s}/F_{s},2))^G\to H^3(G,K_2(A_s)/K_2(F_s)).
\end{multline}

Disposons \eqref{eq2a} et une partie de \eqref{eq3} en croix: pour faire tenir le diagramme
dans la page, on note $M=K_2(A_s)/K_2(F_s)$, $C=\Z(A/F,2)$ et $\Q/\Z(2)=2$.
\[\xymatrix{
&&H^3(F,2)\ar[d]\\
&& H^3(A,2)\ar[d]\\
0 \ar[r]& H^2(G,M)\ar[r]& \bH^5(F,C)
\ar[r]\ar[d]& H^3(A_s,2)^G\ar[r] & H^3(G,M)\\
&&H^4(F,2)\ar[d]\\
&& H^4(A,2)
}\]

Un argument simple  \cite[Lemma of the 700th, p. 142]{700}  montre  alors que l'homologie du complexe
\[0\to H^2(G,M)\to H^4(F,2)\to H^4(A,2)\]
en $H^2(G,M)$ (resp. en $H^4(F,2)$) est isomorphe \`a l'homologie du complexe
\[H^3(F,2)\to H^3(A,2)\to H^3(A_s,2)^G\to H^3(G,M)\]
en $H^3(A,2)$ (resp. en $H^3(A_s,2)^G$). La conclusion en suit.
\end{proof}

Soit $V$ une  vari\'et\'e lisse g\'eom\'etriquement int\`egre sur  
 $F$. Soit $V^{(1)}$ l'ensemble des points de codimension 1 de $V$.
 Pour tout $x \in V^{(1)}$ le symbole mod\'er\'e d\'efinit une suite
 exacte de modules galoisiens
 \[0 \to \K_{2}(  O_{V_{s},x} ) \to \K_{2}(F_{s}(V)) \to F_{s}(x)^{\times}  \to 0, \]
d'o\`u une suite exacte de groupes de cohomologie
\begin{multline}\label{eq6.1}
H^2(G, K_{2}(O_{V_{s},x} )/K_2(F_s))  \to H^2(G, K_{2}(F_{s}(V))/K_2(F_s))\\
 \to H^2(G,\  F_{s}(x)^{\times}).
\end{multline}

Sur un corps de caract\'eristique z\'ero et de dimension cohomologique 1,
l'\'enonc\'e suivant est \'etabli dans \cite[Prop. 8.4]{ctv}.
Il sert ici, sur un corps global, \`a d\'emontrer
la proposition \ref{troisdef}.

\begin{prop}\label{avecbk}
Soit $V$ une  vari\'et\'e lisse g\'eom\'etriquement int\`egre sur un corps
 $F$.
Notons 
$$ \varepsilon : H^3_{\nr}(V,\Q/\Z(2)) \to H^3_{\nr}( V\times_{F}F_{s},\Q/\Z(2))$$
et
$${\mathcal N}(V) = 
\Ker\Big(H^2(G, K_{2}(F_{s}(V))/K_2(F_s)) \to H^2(G,\bigoplus_{x \in V^{(1)} } F_{s}(x)^{\times})   \Big).$$
Les suites exactes de la proposition \ref{p2.1} se globalisent  en une suite exacte
  $$ H^3(F,\Q/\Z(2)) \to  \Ker(\varepsilon)  \to {\mathcal N}(V) \to H^4(F,\Q/\Z(2)).$$
\end{prop}

\begin{proof} Il suffit de mettre en regard les suites exactes de la proposition
 \ref{p2.1} relatives \`a $F(V)$ et aux $O_{V,x} $, puis de tenir compte de \eqref{eq6.1}.
\end{proof}

\begin{theo}\label{cdFleq1}
Soit  $F$ un corps parfait  de dimension cohomologique $\leq 1$.
Soient $\bar F$ une cl\^oture
 alg\'ebrique de $F$ et $G={\rm Gal}({\bar F}/F)$.
Soit $V$ une $F$-vari\'{e}t\'{e} projective, lisse, g\'{e}om\'{e}triquement int\`{e}gre.
Notons   $\ovV=V\times_{F}\bar F$.
On a
alors une suite exacte 
\begin{multline*}0 \to \Ker \left(CH^2(V) \to CH^2(\ovV)^G\right)  \to  H^1(G, \bH^3_{\et}(\ovV,
\Z(2))     )
\\
\to\Ker\left( H^3_\nr(V,\Q/\Z(2)) \to H^3_\nr(\ovV,\Q/\Z(2))\right)
\\\to
  \Coker \left(CH^2(V) \to CH^2(\ovV)^G\right) \to 0.
\end{multline*}
\end{theo}

\begin{proof} Consid\'erons le diagramme commutatif 
\[\begin{CD}
0 &\to& CH^2(V) @>>> \bH^4_\et(V,\Z(2)) @>>> H^0(V,\H^3(\Q/\Z(2))) &\to& 0\\
&& @VVV @VVV @VVV\\
0 &\to& CH^2(\ovV)^G @>>> \bH^4_\et(\ovV,\Z(2))^G @>>> H^0(\ovV,\H^3(\Q/\Z(2)))^G
\end{CD}\]
o\`u les lignes exactes proviennent de \eqref{eqLK}. 
Dans la suite spectrale de Hochschild-Serre (\`a laquelle on peut penser comme suite spectrale
d'hypercohomologie du complexe de $G$-modules $\Z(2)_{\ovV}$)
\[E_2^{p,q}=H^p(G,\bH^q_\et(\ovV,\Z(2)))\Rightarrow \bH^{p+q}_\et(V,\Z(2))\]
on a $E_2^{p,q} =0$ pour $p>2$ puisque, par hypoth\`ese, $\cd(G)\le 1$. 
D'apr\`es \cite[Thm. 1.8 et Thm. 2.2]{ctraskind} et  \cite{grossuwa}
(voir aussi \cite[Prop. 4.17 et Rem. 4.18 ]{kahncycle}), sur le corps alg\'ebriquement clos
$\overline{F}$,
les groupes $H^0(\ovV , \K_{2})\simeq \bH^2_{\et}(\ovV,\Z(2))$
 et $H^1(\ovV , \K_{2}) \simeq \bH^3_{\et}(\ovV,\Z(2))$
sont chacun extension d'un groupe de torsion d'exposant fini
par un groupe divisible.
 Il en r\'esulte $E_2^{2,q}=0$ pour $q=2,3$.
 Ainsi la fl\`eche verticale centrale dans ce diagramme est surjective de noyau
$H^1(G,\bH^3_\et(\ovV,\Z(2)))$, 
et le th\'eor\`eme \ref{cdFleq1}
r\'esulte du lemme du serpent.
\end{proof}

\begin{rema}\label{r2.1}
Cet argument donne de plus une suite exacte
\begin{multline*}
H^3_\nr(V,\Q/\Z(2))\to H^3_\nr(\ovV,\Q/\Z(2))^G\\
\to H^1(F,CH^2(\ovV))\to
H^1(F,\bH^4_\et(\ovV,\Z(2))).
\end{multline*}
\end{rema}

\begin{rema}
En  caract\'eristique z\'ero, la proposition \ref{avecbk} (pour $cd(k)\le 1$) et
le th\'eor\`eme \ref{cdFleq1} sont
\'etablis dans \cite[prop. 8.4 et thm. 8.5]{ctv}; la premi\`ere est utilis\'ee pour d\'emontrer le second.  
On aurait pu suivre la m\^eme m\'ethode ici,  mais on n'e\^{u}t obtenu le r\'esultat qu'\`a la $p$-torsion pr\`es. 
\end{rema}

\subsection{Descente galoisienne sur un corps fini}
\label{varfinies}

Pour tirer du th\'{e}or\`{e}me \ref{cdFleq1}
des cons\'{e}quences pratiques, il faut contr\^{o}ler le module galoisien
$ \bH^3_{\et}(\ovV,\Z(2))  \simeq H^1(\ovV , \K_{2}) $. 
Le th\'{e}or\`{e}me suivant regroupe des r\'{e}sultats de Raskind et du premier auteur  \cite[Thm. 2.2]{ctraskind}
pour $l\neq p$ et de 
Gros et Suwa \cite[\S 3]{grossuwa} pour $l=p$.  

Pour $V$ projective et lisse sur un corps $F$
d'exposant caract\'eristique $p$,
les groupes $H^{i}(\ovV,\Z_{l}(j))$ et  $H^{i}(\ovV,\Q_{l}/\Z_{l}(j))$
ci-dessous utilis\'es sont pour $l\neq p$ les groupes de cohomologie \'etale
$H^{i}_{\et}(\ovV,\Z_{l}(j))$
 et  $H^{i}_{\et}(\ovV,\Q_{l}/\Z_{l}(j))$.
Pour $l=p$, ce sont ceux d\'efinis par Gros et Suwa dans \cite{grossuwa}.
 
Rappelons que le groupe $\bigoplus_{l } H^3(\ovV,\Z_{l}(2))\{l\}$ est   fini.
Pour la partie premi\`ere \`a $p$, c'est d\^u   
\`a Gabber  \cite{gabber}.
Pour la partie $p$-primaire, c'est d\^u \`a 
 Illusie et Raynaud \cite[p.~194]{illusieraynaud}).

\begin{theo}[\cite{ctraskind,grossuwa}] \label{CTRask} 
Soient  $F$ un corps parfait d'exposant  caract\'e\-ris\-tique~$p$,
  $\bar F$ une cl\^oture
 alg\'ebrique de $F$ et $G={\rm Gal}({\bar F}/F)$.
Soit $V$ une $F$-vari\'{e}t\'{e} projective, lisse, g\'{e}om\'{e}triquement int\`{e}gre.
Soit $M=M(\ovV)$ le module galoisien  fini
$\bigoplus_{l} H^3(\ovV,\Z_{l}(2))\{l\}$.
\begin{itemize}
\item[(a)] Il existe une suite exacte naturelle
$$0 \to D \to \bH^3_{\et}(\ovV,\Z(2)) \to M \to 0,$$
o\`u le  groupe $D$ est divisible.
\item[(b)] Pour tout premier $l$, il existe un isomorphisme naturel 
$$H^2(\ovV, \Q_{l}/\Z_{l}(2)) \iso \bH^3_{\et}(\ovV,\Z(2))\{l\}.$$
\item[(c)] Pour tout premier $l$, il existe un isomorphisme naturel
de groupes divisibles
$$H^2(\ovV, \Q_{l}/\Z_{l}(2))_\div \iso D\{l\}$$
\end{itemize}
\end{theo}

On a  la proposition  
 (cf. \cite[Prop. 4.1]{grossuwa}) :

\begin{prop} 
Soit  $\F$ un corps fini. Soit $V$ une $\F$-vari\'{e}t\'{e} projective, lisse, g\'{e}om\'{e\-}tri\-quement int\`{e}gre.
Soit $M$ 
le module galoisien fini 
$$ \bigoplus_{l} H^3(\ovV,\Z_{l}(2))\{l\}.$$
On a un isomorphisme de groupes finis :
$$H^1(G,\bH^3_{\et}(\ovV,\Z(2)) ) \iso H^1(G,M).$$
\end{prop}

\begin{proof}
De la suite exacte du th\'eor\`eme \ref{CTRask}~(a)  on tire la suite exacte
$$  H^1(G,D)  \to H^1(G,\bH^3_{\et}(\ovV,\Z(2))) \to H^1(G,M) \to H^2(G,D),$$
o\`u $D$ est un groupe divisible. Ceci implique d\'ej\`a $H^2(G,D)=0$
et que la fl\`eche $H^1(G,D_{\tors}) \to H^1(G,D)$ est surjective.
D'apr\`es le th\'eor\`eme  \ref{CTRask}~(c), on a un isomorphisme
de modules galoisiens
$$ \bigoplus_{l} H^2(\ovV, \Q_{l}/\Z_{l}(2))_\div \iso D_{\tors}.$$

 Pour tout $l$, on a un isomorphisme 
 $$H^2(\ovV,  \Z_{l}(2))\otimes \Q/\Z \iso H^2(\ovV, \Q_{l}/\Z_{l}(2))_\div$$
 (cf. \cite[p.~ 781]{ctss}), d'o\`u
 \begin{multline*}H^1(G,H^2(\ovV,  \Z_{l}(2)))\otimes \Q/\Z\simeq H^1(G,H^2(\ovV,  \Z_{l}(2))\otimes \Q/\Z)\\
  \iso H^1(G,H^2(\ovV, \Q_{l}/\Z_{l}(2))_\div).
 \end{multline*}
 
 On a 
 $H^1(G,H^2(\ovV, \Q_{l}/\Z_{l}(2))_\div)=0$
 pour tout premier $l$. 
En effet, ce groupe est un groupe de coinvariants
 d'un groupe divisible, donc est divisible. Pour $l\neq p$, sa finitude, et donc sa
nullit\'e, r\'esulte du fait que le groupe de coinvariants
$H^2_{\et}(\ovV,\Z_{l}(2))_{G}$
est fini (Deligne, voir \cite[\S 2.1]{ctss}). 

Par d\'efinition $H^2(\ovV, \Q_{p}/\Z_{p}(2))=H^0(\ovV,\nu_{\infty}(2))$.
La finitude du groupe des coinvariants de ce dernier groupe sous
l'action de $G$ est \'etablie, suivant O. Gabber, dans \cite[\S 2.2, Formule (34)]{ctss}.
  \end{proof}

La combinaison de ce r\'esultat avec le th\'eor\`eme \ref{cdFleq1} donne  :

\begin{theo}\label{hauptsatz}
Soit $\F$ un corps fini.
Soit $\ovF$ une cl\^{o}ture alg\'{e}brique de $\F$, et soit $G$ le groupe de Galois de $\ovF$
sur
$\F$. Soit $V$ une $\F$-vari\'{e}t\'{e} projective, lisse, g\'{e}om\'{e}triquement int\`{e}gre.
Soit $M$ le module galoisien fini
$\bigoplus_{l} H^3(\ovV,\Z_{l}(2))\{l\}$.
On a alors une suite exacte 
\begin{multline*}
0 \to \Ker \left(CH^2(V) \to CH^2(\ovV)\right)  \to  H^1(G,M) \\
\to  \Ker\left(H^3_\nr(V,\Q/\Z(2)) \to H^3_\nr(\ovV,\Q/\Z(2))\right)\\
\to   \Coker\left(CH^2(V) \to CH^2(\ovV)^G\right) \to 0.
\end{multline*}
\end{theo}

  Si l'on remplace le corps $\F$ par un corps de fonctions d'une variable sur le corps des complexes,
on a un \'enonc\'e analogue \cite[Th\'eor\`eme 8.7]{ctv}, mais  il faut alors supposer que le groupe de
cohomologie coh\'erente $H^2(V,O_{V}) $ est nul.

\begin{cor}\label{brauertrivial}
Soit $\F$ un corps fini de caract\'eristique $p$.
Soit $\ovF$ une cl\^{o}ture alg\'{e}brique de $\F$, et $G=\Gal(\ovF/\F)$.
Soit $V$ une $\F$-vari\'{e}t\'{e} projective et lisse, g\'{e}o\-m\'{e\-}tri\-quement
int\`{e}gre.
\begin{itemize}
\item[(a)] Si $\bigoplus_{l} H^3(\ovV,\Z_{l}(2))\{l\}=0$, on a une suite exacte
\begin{multline*}
0 \to CH^2(V) \to CH^2(\ovV)^G \to H^3_\nr(V,\Q/\Z(2)) \to
H^3_\nr(\ovV,\Q/\Z(2))^G\\
\to H^1(\F,CH^2(\ovV))\to
H^1(\F,\bH^4_\et(\ovV,\Gamma(2))).
\end{multline*}
\item[(b)] Si $V$ est une vari\'{e}t\'{e} g\'e\-o\-m\'e\-tri\-que\-ment rationnelle,  on a 
une suite exacte
$$0 \to CH^2(V)[1/p] \to CH^2(\ovV)^G[1/p] \to H^3_\nr(V,\Q/\Z(2))[1/p] \to 0$$
o\`u pour tout groupe ab\'elien $A$ on note $A[1/p]:= A\otimes_\Z\Z[1/p]$ .
\end{itemize}
\end{cor}

\begin{proof}
  L'\'{e}nonc\'{e} (a) est une cons\'{e}quence imm\'{e}diate
du th\'eo\-r\`eme~\ref{hauptsatz} et 
de la remarque \ref{r2.1}.
Soit $l$ premier, $l\neq p$.
Si  $\ovV$ est rationnelle,   les invariants birationnels
 $H^3_\nr(\ovV,\Q_{l}/\Z_{l}(2))$ et $\Br(\ovV)\{l\}$ sont nuls.
 Comme montr\'e par Grothendieck \cite[(8.9)]{BrIII},
 il r\'esulte de la suite de Kummer que
le quotient de $\Br(\ovV)\{l\}$ par son sous-groupe divisible maximal
 est le groupe fini  $H^3_{\et}(\ovV,\Z_{l}(1))\{l\}$. Pour $\ovV$ rationnelle,
 on a donc  $H^3_{\et}(\ovV,\Z_{l}(1))\{l\}=0$ et donc $H^3_{\et}(\ovV,\Z_{l}(2))\{l\}=0$.
L'\'enonc\'e b) r\'esulte alors du  th\'eo\-r\`eme~\ref{hauptsatz}.
  \end{proof}

\begin{rema}
La question de la surjectivit\'e de la fl\`eche  $ CH^2(V) \to CH^2(\ovV)^G$
sur un corps de base fini
avait \'et\'e pos\'ee par
T.~Geisser. A. Pirutka \cite{pirutka}
vient d'exhiber des exemples de vari\'et\'es projectives et lisses sur un corps fini,
 g\'e\-o\-m\'e\-tri\-que\-ment rationnelles,  de dimension 5,  pour
lesquelles les   groupes 
 $H^3_\nr(V,\Q/\Z(2)) $ et  $\Coker\left(CH^2(V) \to CH^2(\ovV)^G\right) $
 sont non nuls.
 Le corollaire  \ref{brauertrivial} (b)  montre directement que, pour les vari\'et\'es 
g\'e\-o\-m\'e\-tri\-que\-ment rationnelles,  ces deux groupes sont isomorphes  
 \`a la $p$-torsion pr\`es.  
\end{rema}

\bigskip

Soit $V/\F$ une surface projective et lisse, g\'e\-o\-m\'e\-tri\-que\-ment int\`egre.
Comme rappel\'e dans la proposition \ref{ctssgros}, on a
  $H^3_\nr(\ovV,\Q/\Z(2))=0$ et  
 $H^3_\nr(V,\Q/\Z(2))=0$.
Le th\'eor\`eme \ref{hauptsatz} donne donc une suite exacte
$$ 0 \to H^1(G,M) \to CH^2(V) \to CH^2(\ovV)^G \to 0.$$

Comme une telle surface poss\`ede un z\'ero-cycle de degr\'e 1, 
cela donne aussi
une suite exacte
$$ 0 \to H^1(G,M) \to A_{0}(V) \to A_{0}(\ovV)^G \to 0,$$
  o\`u $A_{0}(V) \subset CH_{0}(V)$ est le groupe des z\'ero-cycles de degr\'e z\'ero
modulo \'equivalence rationnelle. 

Le th\'eor\`eme de Ro\v\i tman donne un isomorphisme 
$$A_{0}(\ovV) \iso Alb_{X}(\ovF)$$
entre le groupe de Chow des z\'ero-cycles de degr\'e z\'ero modulo l'\'equivalence
rationnelle et le groupe des points g\'eom\'etriques de la vari\'et\'e d'Albanese.
Par ailleurs on v\'erifie que le module galoisien $M$
est  isomorphe
 au dual  de Cartier $D$ de la torsion du groupe de N\'eron-Severi $NS$ de $\ovV$.
On obtient alors une suite exacte $$ 0 \to H^1(\F, D)) \to A_{0}(V) \to Alb_{X}(\F) \to 0. $$

On retrouve ainsi, dans le cas des surfaces, un th\'eor\`eme de K.~Kato et S.~Saito (\cite[Prop. 9.1]{KS},
voir aussi \cite{ctraskind2}).

 \section{Z\'ero-cycles   sur les corps globaux de caract\'eristique positive}\label{globalpositif}
 
 \subsection{Conjectures sur les z\'ero-cycles : rappels}\label{rappelscyclesglobalpositif}
 
  Soit $\F$ un corps fini, $C/\F$ une courbe projective, lisse, g\'eom\'etriquement connexe,
de corps des fonctions $K=\F(C)$. Pour $v$ point ferm\'e de $C$, notons 
$K_{v}$ le compl\'et\'e de $K$ en $v$, c'est-\`a-dire le corps des fractions
de l'anneau local compl\'et\'e $\hat{O}_{C,v}$.

Soit $X/\F$ une vari\'et\'e projective, lisse, g\'eom\'e\-tri\-quement connexe de dimension $d+1$, \'equip\'ee
d'un morphisme dominant $p : X \to C$ de fibre g\'en\'erique $V/K$ lisse et g\'eom\'e\-tri\-quement int\`egre.
Soit $l$ premier diff\'erent de la caract\'eristique de $\F$. Notons $V_{v}=V\times_{K}K_{v}$.

Comme expliqu\'e dans \cite[\S 1 et \S 2]{ctseattle},  on a un diagramme commutatif de complexes :

\[\begin{CD}
 H^{2d}_{\et}(V,\mu_{l^n}^{{\otimes} d}) @>>> \displaystyle \prodr_{v \in
C^{(1)} } H^{2d}_{\et}(V_{v},\mu_{l^n}^{{\otimes} d} ) @>>> \Hom(H^{2}_{\et}(V,\mu_{l^n}),   
\Q_{l}/\Z_{l})\\ 
@AA{cl}A @AA{cl}A @AAA\\ 
CH_{0}(V) @>>>\displaystyle \prod_{v \in C^{(1)} }
CH_{0}(V_{v})@>>>  \Hom(\Br(V)[l^n], \Q_{l}/\Z_{l}).
\end{CD}\]

Dans ce diagramme, le  terme m\'edian  du complexe sup\'erieur est un produit direct restreint :
on ne prend que les familles  $\{\xi_{v}\}$ telles que pour presque tout $v$, $\xi_{v}$ est dans
l'image de 
$H^{2d}_{\et}(X\times_{C}Spec (\hat{O}_{C,v}), \mu_{l^n}^{{\otimes} d} )$. C'est une suite
exacte (S.~Saito
\cite{saitodual}).
Les fl\`eches $\rm cl$ sont les applications cycle en cohomologie \'etale. La fl\`eche
$\prod_{v \in C^{(1)} } CH_{0}(V_{v}) \to  \Hom(\Br(V)[l^n], \Q_{l}/\Z_{l})$ est induite par
les accouplement naturels
$$ CH_{0}(V_{v}) \times \Br(V_{v}) \to \Br(K_{v}) \iso \Q/\Z$$
qui a un z\'ero-cycle $z_{v}$ et un \'el\'ement $A \in \Br(V_{v})$ associe $\inv_{v}(A(z_{v})) \in \Q/\Z$.

\`A la suite de divers travaux (\cite{cts81}, \cite{katosaito}, \cite{saito}),
le premier auteur a formul\'e  la conjecture suivante \cite[Conjecture 2.2]{ctseattle}.

\begin{conj} \label{congenzerocycletotale}  
Pour chaque $v \in C^{(1)}$, soit $z_{v} \in CH_{0}(V_{v})$.
Supposons que pour tout \'el\'ement  
$A \in \Br(V)\{l\}$, on ait $\sum_{v} \inv_{v}(A(z_{v}))=0$. 
Alors pour tout $n>0$ il existe   $z_{n} \in CH_{0}(V)$ 
tel que pour tout $v$ on ait $cl(z_{n})=cl(z_{v}) \in H^{2d}_{\et}(V_{v},\mu_{l^n}^{{\otimes} d} ).$
\end{conj}

Elle a comme cas particulier l'\'enonc\'e suivant (voir \cite[Statement (${\rm M}_{l}^*)$, p.~400]{saito}) :

\begin{conj} \label{congenzerocycle} 
S'il existe une famille $\{z_{v}\}_{v \in C^{(1)}}$, de z\'ero-cycles locaux de
degr\'es premiers \`a $l$  telle que pour tout \'el\'ement  
$A \in \Br(V)\{l\}$ on ait $\sum_{v }    \inv_{v}(A(z_{v}))=0$ alors il existe un z\'ero-cycle de
degr\'e premier \`a $l$ sur~$V$.
\end{conj}

Pour plus de d\'etail sur ces conjectures, que l'on peut formuler plus g\'en\'era\-lement pour
 toute vari\'et\'e projective et lisse sur un corps global $K$,
en particulier sur un corps de nombres, on consultera 
l'introduction de  \cite{wittenberg}.

 \medskip

Le th\'eor\`eme suivant \cite[Prop. 3.2]{ctseattle}   \'etend un th\'eor\`eme de Shuji Saito
\cite[Cor. (8-6)]{saito}.

\begin{theo}[\cite{ctseattle}]   \label{brauermaninF(C)}
Soit $\F$ un corps fini, $C/\F$ une courbe projective, lisse, g\'eom\'etriquement connexe,
de corps des fonctions $K=\F(C)$.
Soit $X/\F$ une vari\'et\'e projective, lisse, g\'eom\'e\-tri\-quement connexe de dimension $d$, \'equip\'ee
d'un morphisme dominant $p : X \to C$ de fibre g\'en\'erique $V/K$ lisse et g\'eom\'e\-tri\-quement int\`egre.
Soit $l$ premier diff\'erent de la caract\'eristique de $\F$.
Si l'application cycle 
$$CH^{d-1}(X) \otimes \Z_{l} \to H^{2d-2}_{\et}(X,\Z_{l}(d-1))$$ 
est surjective, les  conjectures \ref{congenzerocycletotale} et \ref{congenzerocycle} valent
pour $V$.
\end{theo}

On ne conna\^{i}t pas de contre-exemple \`a l'hypoth\`ese de surjectivit\'e
dans le th\'eor\`eme \ref{brauermaninF(C)}. 
Sur la cl\^{o}ture alg\'ebrique d'un corps fini,  voir la d\'emonstration
de la proposition \ref{dim3algclos}.

Sur un corps fini, 
 du th\'eor\`eme de Lefschetz faible
 on d\'eduit ais\'ement que la surjectivit\'e de
 $$CH^{2}(X) \otimes \Z_{l} \to H^{4}_{\et}(X,\Z_{l}(d-1))$$
 pour toute vari\'et\'e projective et lisse $X$ de dimension 3
 impliquerait la surjectivit\'e de
$$CH^{d-1}(X) \otimes \Z_{l} \to H^{2d-2}_{\et}(X,\Z_{l}(d-1))$$ 
pour toute vari\'et\'e projective et lisse $X$ de dimension $d>3$.

\subsection{Le cas des solides fibr\'es en coniques}
  
Commen\c cons par un \'enonc\'e g\'en\'eral sur les solides.

\begin{theo}\label{corollairebrauermaninF(C)}
Soit $\F$ un corps fini, $C/\F$ une courbe projective, lisse, g\'eo\-m\'e\-tri\-quement connexe,
de corps des fonctions $K=\F(C)$.
Soit $X/\F$ une vari\'et\'e projective, lisse, g\'eom\'etriquement connexe de dimension~$3$, \'equip\'ee
d'un morphisme dominant $p : X \to C$ de fibre g\'en\'erique une surface $V/K$ lisse et g\'eom\'etriquement int\`egre. Soit $l$ premier diff\'erent de la caract\'eristique de $\F$.
 
  On suppose que : 
  \begin{thlist}
\item La conjecture de Tate \cite{tate} vaut pour les diviseurs sur $X$, c'est-\`a-dire  
  que l'application cycle
$$ CH^1(X) \otimes \Q_{l} \to H^2_{\et}(X,\Q_{l}(1))$$
est surjective. 

 \item  Le groupe $H^3_\nr(X,\Q_{l}/\Z_{l}(2))$ est divisible.
 \end{thlist}
 
Alors  les conjectures \ref{congenzerocycletotale} et \ref{congenzerocycle} valent pour $V$.
\end{theo}

\begin{proof}
D'apr\`es la proposition \ref{brauerfini}, comme la vari\'et\'e $X$ est de dimension~3,
 l'application cycle
$$ CH^{2}(X) \otimes \Z_{l} \to H^{4}_{\et}(X,\Z_{l}(2)) $$
a son conoyau fini.
 D'apr\`es le
th\'eor\`eme 
\ref{Hauptsatz}, la torsion de ce conoyau, donc le conoyau lui-m\^eme,
est le quotient de
 de $H^3_\nr(X,\Q_{l}/\Z_{l}(2))$ par son sous-groupe divisible maximal,
et par l'hypoth\`ese (ii)  ce quotient est nul. L'application cycle ci-dessus 
est donc surjective, et
on peut alors appliquer le th\'eor\`eme
\ref{brauermaninF(C)}.
\end{proof}

Le lemme suivant rassemble des r\'esultats bien connus.

\begin{lem}\label{coniques}
(i) Soit $K$ un corps et $C$ une $K$-conique lisse. L'application naturelle
$\Br(K) \to \Br(C)$ est surjective, et son noyau est d'ordre au plus 2.

(ii) Soit $A$ un anneau de valuation discr\`ete de corps des fractions $K$ et de corps r\'esiduel $\kappa$
de caract\'eristique diff\'erente de $2$.Toute conique lisse $C/K$ admet un mod\`ele r\'egulier
$Y \subset \P^2_{A}$ donn\'e par une \'equation homog\`ene $x^2-ay^2-bt^2=0$
avec $a \in A^{\times}$ et $b$ de valuation $0$ ou $1$.
La fibre sp\'eciale de $Y/A$ est soit

(ii.a)  une $\kappa$-courbe propre, lisse, g\'eom\'etriquement connexe de genre z\'ero 

ou 

(ii.b)  une $\kappa$-courbe int\`egre qui sur une extension au plus quadratique $\lambda/\kappa$
est la r\'eunion disjointe de deux droites  se coupant transversalement en
un $\kappa$-point.

(iii) Soit $\alpha \in \Br(K)\{l\}$ avec $l \in A^{\times}$. Supposons que 
l'image de $\alpha$ dans $\Br(C)$ appartient \`a $\Br(Y) \subset \Br(C)$.
Alors le r\'esidu $\partial _{A}(\alpha)  \in H^1(\kappa,\Q_{l}/\Z_{l})$
est  nul dans le cas (ii.a), et appartient \`a $H^1(\lambda/\kappa,\Z/2)\hookrightarrow  \Z/2$
dans le cas (ii.b). Il est en particulier nul si $l\neq 2$.
\cqfd
\end{lem}

\begin{prop}\label{Tatemonteauxconiques}
Soient $\F$ un corps fini de caract\'eristique $p$ dif\-f\'e\-rente de $2$,  
et $ S,X$ des vari\'et\'es projectives, lisses, g\'eom\'etriquement connexes
sur $\F$, 
$X \to S $ un morphisme dominant de    fibre g\'en\'erique $X_{\eta}$ une conique lisse
sur le corps $\F(S)$.
Si la conjecture de Tate pour les diviseurs et  le nombre premier $l \neq p$ vaut pour $S$, alors elle vaut pour $X$.
\end{prop}
\begin{proof}
 Soit $\alpha \in \Br(X)\{l\}$. 
 Soit $l$ un nombre premier impair. 
Sa restriction \`a $\Br(X_{\eta})$ est l'image
 d'un \'el\'ement $\beta \in \Br(\F(S))$. 
 Soit $x$ un point de codimension $1$ sur $S$, et soit $A=O_{S,x}$ l'anneau 
 de valuation discr\`ete d\'efini par son anneau local.
 Soit $Y/A$ un mod\`ele de $X_{\eta}/\F(S)$ comme dans le lemme \ref{coniques} (ii).
 Les $A$-sch\'emas r\'eguliers propres $Y/A$ et $X\times_{S}\Spec(A)$ 
  ont des fibres g\'en\'eriques isomorphes. Par la puret\'e du groupe
 de Brauer sur les sch\'emas r\'eguliers de dimension 2 \cite[Prop. 2.3]{BrII}, ceci implique que
 $\alpha \in \Br(X_{\eta})$ appartient \`a $\Br(Y)$. Il r\'esulte alors du lemme
 \ref{coniques} (iii) que  $\partial_{A}(\alpha)$ est nul si on est dans le cas (ii.a)
 et ne peut prendre que l'une de deux valeurs si on est dans le cas (ii.b).
 Partant d'une \'equation $x^2-ay^2-bt^2=0$ pour $X_{\eta}$ sur le corps
 $\F(S)$, on voit qu'il n'y a qu'un ensemble $T$  fini de points de codimension 1
 de $S$ o\`u l'on ne peut pas prendre un mod\`ele du type (ii.a).
D'apr\`es \cite[Th\'eor\`eme (6.1)]{BrIII}, on a la suite exacte
$$0 \to \Br(S)\{l\} \to \Br(\F(S)\{l\} \to \bigoplus_{x \in X^{(1)}}     H^1(\F(S),\Q_{l}/\Z_{l}).$$
La conjecture de Tate pour la surface $S$ dit que le groupe $ \Br(S)\{l\}$
est fini.  Le sous-groupe des \'el\'ements de $ \Br(\F(S)\{l\}$
dont les r\'esidus hors de $T$ sont nuls et dont les r\'esidus aux points de $T$
appartiennent au groupe fini $$H^1(\lambda/\kappa,\Z/2)\hookrightarrow  \Z/2$$
est donc fini. Ceci implique que $\Br(X)\{l\}$ est fini, et donc la conjecture de Tate
vaut pour les diviseurs sur $X$ et le nombre premier $l$.
  \end{proof}

\begin{rema}
 L'argument d\'evelopp\'e dans la proposition pr\'ec\'edente a une port\'ee plus
g\'en\'erale.  Soient $\F$ un corps fini de caract\'eristique $p$ dif\-f\'e\-rente de $2$,  
et $ S,X$ des vari\'et\'es projectives, lisses, g\'eom\'etriquement connexes
sur $\F$, puis
$X \to S $ un morphisme dominant de    fibre g\'en\'erique $V=X_{\eta}/\F(S)$
lisse et g\'eom\'etriquement int\`egre. Supposons que pour un premier  $l \neq p$
l'application 
$\Br(\F(S))\{l\}  \to \Br(V)\{l\}$
est surjective (c'est le cas par exemple si $V$ est une intersection compl\`ete lisse
de dimension au moins $3$ dans un espace projectif).
Supposons le groupe $\Br(S)\{l\}$ fini. En utilisant  \cite[Prop. (4.3)]{tate},
le th\'eor\`eme de puret\'e pour le groupe de Brauer, et le comportement
des r\'esidus d'un \'el\'ement de $\Br(\F(S))\{l\} $ sur $S$ par passage \`a $X$,
on montre que le groupe $\Br(X)$ est fini. 
\end{rema}

\begin{theo}\label{parasugeneralise}
Soient $\F$ un corps fini de caract\'eristique $p$ dif\-f\'e\-rente de $2$,  
et $C,S,X$ des vari\'et\'es projectives, lisses, g\'eom\'etriquement connexes
sur $\F$, de dimensions respectives $1,2,3$, \'equip\'ees de morphismes dominants
$X \to S \to C$, la fibre g\'en\'erique  de $X\to S$ \'etant une conique lisse, 
la fibre  g\'en\'erique  de $S \to C$ \'etant une courbe lisse g\'eom\'etriquement int\`egre.
Soient  $K=\F(C)$ et $V$  la $K$-surface fibre g\'en\'erique
de l'application compos\'ee $X \to C$.

Si  la conjecture de Tate pour les diviseurs vaut  sur la $\F$-surface $S$, alors, pour tout $l\neq p$,
les conjectures \ref{congenzerocycletotale} et \ref{congenzerocycle} valent pour la $K$-surface $V$.
\end{theo}

 \begin{proof}
 D'apr\`es la proposition \ref{Tatemonteauxconiques}, la conjecture de Tate  en codimension $1$  vaut pour $X$.
D'apr\`es le th\'eor\`eme \ref{H3parimalasuresh} (Parimala et Suresh),
on a  $H^3_\nr(X,\Q_{l}/\Z_{l}(2))=0$ pour tout premier $l \neq p$.
On peut alors appliquer le  th\'eor\`eme \ref{corollairebrauermaninF(C)}.
\end{proof}

\begin{cor} \label{pasulocglob}
Soient $\F$ un corps fini de caract\'eristique $p$ dif\-f\'e\-rente de $2$,  
et $C,S,X$ des vari\'et\'es projectives, lisses, g\'eom\'etriquement connexes
sur $\F$, de dimensions respectives $1,2,3$, \'equip\'ees de morphismes dominants
$X \to S \to C$, la fibre g\'en\'erique  de $X\to S$ \'etant une conique lisse
 et la fibre g\'en\'erique  de $S\to C$ \'etant une conique lisse.
  Alors pour tout $l\neq p$, les conjectures \ref{congenzerocycletotale} et \ref{congenzerocycle} valent pour $V/\F(C)$,
fibre g\'en\'erique de l'application compos\'ee $X \to C$.

\end{cor}

 \begin{proof}
En utilisant les propri\'et\'es des coniques et le th\'eor\`eme de Tsen, on voit qu'il existe une extension finie
$\F'/\F$,  une courbe $C'$ sur $\F'$  et, sur $\F'$, une application rationnelle dominante
de  ${\mathbb P}^1 \times  {\mathbb P}^1 \times C' $ vers $X'=X\times_{\F}\F'$. 
D'apr\`es  \cite[Prop. 5.2 (b]{tate}, la conjecture de Tate pour  les diviseurs sur $X$ r\'esulte de
la m\^{e}me conjecture   sur ${\mathbb P}^1 \times  {\mathbb P}^1 \times C' $,  laquelle est \'evidente.
Il est clair que la $\F(C)$-vari\'et\'e poss\`ede un z\'ero-cycle de degr\'e~$4$.
L'\'enonc\'e d\'ecoule  alors du th\'eor\`eme \ref{parasugeneralise}.
\end{proof}

 \begin{rema}
La surface $V$ poss\`ede clairement un z\'ero-cycle de degr\'e~$4$.
La conjecture \ref{congenzerocycle}   s'\'ecrit donc ici :
S'il existe une famille $\{z_{v}\}_{v \in C^{(1)}}$ de z\'ero-cycles locaux de
degr\'e 1 telle que pour tout \'el\'ement  
$A \in \Br(V)$ d'ordre impair on ait $\sum_{v }    \inv_{v}(A(z_{v}))=0$,
 alors il existe un z\'ero-cycle de
degr\'e 1 sur~$V$. Cet \'enonc\'e 
est  d\'ej\`a obtenu  par Parimala et Suresh 
dans \cite{parimalasuresh}.
 \end{rema}

 \section{Principe local-global pour $H^3$ non ramifi\'e}\label{global}

 Soient $K$ un corps global, $\Omega$ l'ensemble de ses places.
 Pour $v \in \Omega$, on note $K_{v}$ le compl\'et\'e de $K$ en $v$.
On note $K_{s}$ une cl\^oture s\'eparable de $K$ et $g=\Gal(K_{s}/K)$.
 Pour toute $K$-vari\'et\'e   $V$, on note  $\ovV=V\times_{K}K_{s}$,
 et $V_{v}=V \times_{K}K_{v}$.

\subsection{Questions et conjectures en $K$-th\'eorie alg\'ebrique}\label{conjKtheorie}

Soient $F$ un corps, $F_{s}$ une cl\^oture s\'eparable de $k$,
puis $g=\Gal(F_{s}/F)$. 
Soit $V$ une $F$-vari\'et\'e lisse g\'eom\'etriquement int\`egre.
 On dispose de l'application $g$-\'equivariante d\'efinie par les symboles mod\'er\'es
$$K_{2}F_{s}(V)/K_{2}F_{s} \to \bigoplus_{x \in V^{(1)}}  F_{s}(x)^{\times}.$$
Comme dans la proposition \ref{avecbk}, on note :
$${\mathcal N}(V)  : = \Ker( H^2(g,K_{2}F_{s}(V)/K_{2}F_{s}) \to H^2(g, \bigoplus_{x \in V^{(1)}}  F_{s}(x)^{\times})).$$

Si maintenant $F$ est 
un corps global $K$, on note
$$\cyr{X}{\mathcal N}(V) := \Ker ({\mathcal N}(V) \to \prod_{v} {\mathcal N}(V_{v})).$$

{Soient  $V$ une $K$-vari\'et\'e projective, lisse, g\'eom\'etriquement connexe.

 \begin{conj}\label{questionfinieglobalK} 
 Le groupe $\cyr{X}{\mathcal N}(V)$ est fini. 
    \end{conj}
    
     \begin{qn} \label{questiondim3geomglobalK} 
 Si $\dim V=2$,  le groupe  $\cyr{X}{\mathcal N}(V)$ est-il nul  ?
\end{qn}

\begin{conj}
\label{conjsurfregtglobalK} 
  Si $V$ est une  $K$-surface g\'eom\'etri\-quement r\'egl\'ee,  
   le groupe  $\cyr{X}{\mathcal N}(V)$ est nul.
\end{conj}

La conjecture \ref{conjsurfregtglobalK}  est inspir\'ee de \cite{cts81}, o\`u elle n'est formul\'ee
que pour les surfaces   rationnelles (conjecture B, op. cit.)

 Comme expliqu\'e dans \cite{cts81} (voir aussi la sous-section \ref{surfrat} ci-dessous), cette conjecture implique l'essentiel des   conjectures
locales-globales sur les z\'ero-cycles d'une telle surface faites dans \cite{cts81}.
 Pour $X$ une surface fibr\'ee en coniques sur la droite projective sur un corps de nombres,
 la conjecture B de  \cite{cts81}  fut \'etablie par Salberger \cite[Thm. (7.1)]{salberger1}. 
 
 \begin{rema}\label{shaNVquestion}

Des r\'esultats   sur les z\'ero-cycles,  analogues \`a ceux du  th\'eo\-r\`eme \ref{parasugeneralise},
ont \'et\'e obtenus pour les surfaces fibr\'ees en coniques sur une courbe de genre quelconque sur un corps
de nombres (le premier auteur, Frossard, van Hamel, Wittenberg, voir  \cite{wittenberg}).
Dans ces travaux, on fait l'hypoth\`ese que
le groupe de Tate-Shafarevich de la jacobienne de la courbe est fini. 
C'est l'analogue de l'hypoth\`ese que la conjecture de Tate vaut pour la surface $S$ dans le th\'eor\`eme
\ref{parasugeneralise}.

Il conviendrait de v\'erifier
que sans faire  cette hypoth\`ese les articles en question \'etablissent $\cyr{X}{\mathcal N}(V)=0$ 
pour toute telle  surface $V$ fibr\'ee en coniques au-desssus d'une courbe de genre quelconque, 
et donc  la conjecture \ref{conjsurfregtglobalK} pour de telles surfaces.
\end{rema}
 
\subsection{Traduction en termes de $H^3_{\nr}$}\label{traductionKH3}

Pour $V$ une vari\'et\'e projective, lisse, g\'eom\'etriquement connexe sur 
le corps global $K$, on note
$$\cyr{X}H^3_{\nr}(V,\Q/\Z(2)) : = \Ker (H^3_{\nr}(V,\Q/\Z(2)) \to \prod_{v} H^3_{\nr}(V_{v},\Q/\Z(2))).$$

On note
$$ \varepsilon : H^3_{\nr}(V,\Q/\Z(2)) \to H^3_{\nr}( \ovV,\Q/\Z(2)).$$
D\'efinissons  $\cyr{X}\left(\Ker(\varepsilon)\right)$ de la fa\c con \'evidente.

La fl\`eche naturelle  $\cyr{X}\left(\Ker(\varepsilon)\right) \to
\cyr{X}H^3_{\nr}(V,\Q/\Z(2))
$ est un isomorphisme. En effet, si $L$ est une extension r\'eguli\`ere de $K$,
toute classe dans $H^3_{\nr}(V,\Q/\Z(2))$ qui s'annule apr\`es extension
des scalaires de $K$ \`a $L$ s'annule par passage \`a $K_{s}$ 
\footnote{C'est une propri\'et\'e g\'en\'erale des foncteurs commutant
aux limites inductives filtrantes.}.

Notons $\Omega_{\R}$ l'ensemble, \'eventuellement vide, des places r\'eelles de $K$.
On a  :
$$ H^3(K,\Q/\Z(2)) \iso \bigoplus_{v \in \Omega_{\R}} H^3(K_{v},\Q/\Z(2)) =  \bigoplus_{v \in
\Omega_{\R}} \Z/2,$$ et
$$ H^4(K,\Q/\Z(2)) \iso \bigoplus_{v \in \Omega_{\R}} H^4(K_{v},\Q/\Z(2)) =0.$$
 Pour la torsion premi\`ere \`a la caract\'eristique de $K$, voir
 \cite[Thm. I. 4.10]{milneduality}. Si $K$ est un corps de fonctions d'une variable sur
 un corps fini de carac\-t\'e\-ristique~$p>0$, on a $\nu_{n}(2)=0$ pour tout $n>0$
  et donc $H^{r}(K,\Z/p^n(2))=0$
 pour tout entier $r$.

De   la proposition \ref{avecbk}  
on d\'eduit alors :

\begin{prop}\label{troisdef}
Soit $K$ un corps global.
  On a une suite exacte
$$0 \to  A  \to \cyr{X}H^3_{\nr}(X,\Q/\Z(2)) \to \cyr{X}{\mathcal N}(V) \to 0,$$
o\`u  $A=(\Z/2)^s$,
    avec   $s$
    au plus \'egal au nombre de places r\'eelles de $K$ 
    pour lesquelles 
    l'ensemble $V(K_{v})$  des $K_{v}$-point de $V$ est  vide.
\end{prop}

\begin{rema} Soit $K=\Q$. Soit $V/\Q$ avec $V(\R)=\emptyset$.
Le groupe $A$ vaut $\Z/2$ si et seulement si
$-1$ n'est pas une somme de 4 carr\'es dans $\Q(V)$
mais $-1$ est une somme de 4 carr\'es dans $\R(V)$.
Si  $V$ est  une surface, alors
  $-1$ est une somme de 4 carr\'es dans $\R(V)$
(Pfister) et $-1$  une somme de
 $8$ carr\'es dans $\Q(V)$ (Jannsen et le premier auteur).
C'est   un probl\`eme ouvert de savoir
s'il existe une telle surface $V$  sur $\Q$ pour laquelle 
$-1$ n'est pas une somme de 4 carr\'es dans $\Q(V)$.
Voir \`a ce sujet  \cite{jannsensujatha}.
\end{rema}

Soit $V$ une $K$-vari\'et\'e   projective, lisse,
g\'e\-o\-m\'e\-tri\-que\-ment connexe.
D'apr\`es la proposition \ref{troisdef}, les
 questions et conjectures du paragraphe \ref{conjKtheorie}
 sont \'equiva\-lentes aux questions et conjectures suivantes.

  \begin{conj}\label{questionfinieglobal}Le groupe $\cyr{X}H^3_{nr}(V,\Q/\Z(2))$ est fini.
    \end{conj}

 \begin{qn} \label{questiondim3geomglobal} 
Si $\dim V=2$, 
  le groupe $\cyr{X}H^3_{nr}(V,\Q/\Z)(2)$ est-il  fini d'exposant 2,
  nul si $V $ poss\`ede des $K_{v}$-points pour toute place r\'eelle $v$?
\end{qn}

\begin{conj}
\label{conjsurfregtglobal} 
Si $V$ est  une  $K$-surface g\'eom\'etri\-quement r\'egl\'ee,  
    le groupe $\cyr{X}H^3_{nr}(V,\Q/\Z)(2)$ est un groupe fini $A=(\Z/2)^s$,
    avec $s$ au plus \'egal au nombre de places r\'eelles de $K$ 
    telles que $V(K_{v})=\emptyset$.
\end{conj}

\begin{prop}\label{sommedirecte}
Soit $V/K$ une surface projective et lisse sur un corps global $K$. Soit $l$ un premier distinct de la
caract\'eristique de $K$.
L'image de l'application de restriction aux compl\'et\'es de $K$
$$H^3_\nr(V,\Q_{l}/\Z_{l}(2)) \to \prod_{v \in \Omega} H^3_\nr(V_{v},\Q_{l}/\Z_{l}(2))$$ 
appartient \`a  la somme directe $ \displaystyle\bigoplus_{v \in \Omega} H^3_\nr(V_{v},\Q_{l}/\Z_{l}(2))$.
\end{prop}

\begin{proof}  Soit ${\mathcal V}/U$ un mod\`ele projectif et lisse  de $V/K$ au-dessus d'un
ouvert $U={\rm Spec}(O)$ de l'anneau des entiers du corps de nombres $K$, 
ou d'une courbe lisse sur un corps fini, de corps des fractions
le corps global $K$ de caract\'eristique positive.
Soit $\xi \in H^3(k(V),\mu_{l^n}^{\otimes 2})$ une classe non ramifi\'ee sur $V$.
Les points  (en nombre fini) de codimension $1$   de $\mathcal{V}$ 
o\`u le r\'esidu de
$\xi$ est non nul sont situ\'es au-dessus d'un nombre fini de points ferm\'es de~$U$.
Quitte \`a remplacer $U$ par un ouvert non vide, on peut donc
supposer que $\xi$  appartient \`a $H^0({\mathcal V},{\mathcal H}^3(\mu_{l^n}^{\otimes 2}))$.
Pour tout $v \in U$,
son image $\xi_{v} \in H^3_\nr(V_{v},\mu_{l^n}^{\otimes 2})$ appartient donc
\`a $H^0({\mathcal V}\times_{O}O_{v},{\mathcal H}^3(\mu_{l^n}^{\otimes 2}))$.
Or S. Saito et K.~Sato  ont montr\'e  que pour une surface relative projective et lisse
sur un anneau de valuation discr\`ete hens\'elien  excellent, \`a corps r\'esiduel fini,
  ce dernier groupe est nul
(\cite{saitosato2},  cf. \cite[th. 3.16]{ctbbki}).
 \end{proof}

\subsection{Sur un corps global de caract\'eristique positive}

Sur un tel corps, le lien entre les conjectures et questions ci-dessus et celles du \S
\ref{exemplesquestions}  est fourni par la :

\begin{prop}\label{passageH3gen}
Soit $\F$ un corps fini, $C/\F$ une $\F$-courbe projective lisse g\'eom\'etriquement connexe,
$K=\F(C)$ son corps des fonctions, $X$ une $\F$-vari\'et\'e projective lisse g\'eom\'etriquement connexe,
$ f : X \to C$ un morphisme dominant \`a fibre g\'en\'erique lisse g\'eom\'etriquement int\`egre $V/K$.

Pour tout $l$ premier, $l \neq {\rm car.} \F$,
on a une inclusion naturelle
$$\cyr{X}H^3_{\nr}(V,\Q_{l}/\Z_{l}(2)) \hookrightarrow  H^3_{\nr}(X,\Q_{l}/\Z_{l}(2)).$$

\end{prop}

\begin{proof} Une classe dans $$H^3(K(V),\Q_{l}/\Z_{l}(2)) = H^3(\F(X),\Q_{l}/\Z_{l}(2))$$
qui est dans $\cyr{X}H^3_{\nr}(V,\Q_{l}/\Z_{l}(2))$  a tous ses r\'esidus nuls sur $X$. Elle appartient
donc \`a  $H^3_{\nr}(X,\Q_{l}/\Z_{l}(2))$.
\end{proof}

Ainsi :

(a)  La finitude de $ H^3_{\nr}(X,\Q_{l}/\Z_{l}(2))$ (conjecture \ref{questionfinie})
entra\^{\i}ne la finitude de   $\cyr{X}H^3_{\nr}(V,\Q_{l}/\Z_{l}(2))$ (conjecture \ref{questionfinieglobal})
et de $\cyr{X}{\mathcal N}(V)$ (conjecture \ref{questionfinieglobalK}).

(b) Pour  $X$ de dimension $3$, la nullit\'e de  $ H^3_{\nr}(X,\Q_{l}/\Z_{l}(2))$ (question \ref{questiondim3fini})
entra\^{\i}ne la nullit\'e de   $\cyr{X}H^3_{\nr}(V,\Q_{l}/\Z_{l}(2))$ (question \ref{questiondim3geomglobal})
et de $\cyr{X}{\mathcal N}(V)$ (question \ref{conjsurfregtglobalK}).

\medskip

Soient ${\mathcal X} \to S \to C$ des morphismes dominants de vari\'et\'es projectives lisses,
g\'eom\'etriquement connexes de dimensions respectives 3,2,1,  sur un corps fini
$\F$ de caract\'eristique $p$ impaire,  soit $K=\F(C)$ et $$X={\mathcal X}\times_{C} {\rm \Spec} K.$$
 Supposons que la fibre g\'en\'erique de $X\to S$ soit une conique. 
Pour $l \neq p$ un nombre premier, le th\'eor\`eme de Parimala et Suresh (th\'eor\`eme \ref{H3parimalasuresh} ci-dessus)
donne  $H^3_{\nr}(X,\Q_{l}/\Z_{l}(2))=0$.
Pour la surface $V/K=\F(C)$,  on a  donc $ \cyr{X}H^3_{\nr}(V,\Q_{l}/\Z_{l}(2))=0$
et   $\cyr{X}{\mathcal N}(V)=0$.  
On comparera ce r\'esultat \`a la remarque
\ref{shaNVquestion} ci-dessus.

 \subsection{Cas des surfaces rationnelles}\label{surfrat}\
 
 Soient $F$  un corps et $F_{s}$
 une cl\^{o}ture s\'eparable
  de $K$. Soit $g={\Gal}(F_{s}/F)$.
Soit $V$ une $F$-surface projective, lisse, g\'e\-o\-m\'e\-tri\-que\-ment rationnelle.
Le groupe de Picard $\Pic(V\times_{F}F_{s})$
est un $g$-r\'eseau autodual. On note $S$ le $K$-tore de groupe de caract\`eres le
$g$-r\'eseau  $\Pic(V\times_{F}F_{s})$. Soit $A_{0}(V)$ le groupe des z\'ero-cycles de degr\'e
z\'ero  sur $V$ modulo l'\'equivalence rationnelle. L'indice $I(V)$ de $V$ est le pgcd des 
degr\'es des points ferm\'es de $V$.

\begin{theo}[\protect{\cite{blochrat,cts81,ctHilbert90}}]\label{surfratavecpoint}
Supposons $I(V)=1$.
On a alors  une suite exacte naturelle :
$$ 0 \to A_{0}(V) \to H^1(F,S) \to {\mathcal N}(V) \to 0.$$
\end{theo}

Supposons maintenant que $F$ soit un corps global $K$.
L'application diagonale $H^1(K,S) \to \prod_{v} H^1(K_{v},S)$ a son image dans la somme
directe;  on note 
$$ \cyr{CH}^1(K,S) :=\Coker \Big(H^1(K,S) \to \bigoplus_{v} H^1(K_{v},S)\Big).$$

La dualit\'e de Tate-Nakayama pour le $K$-tore $S$
 donne une injection
$$   \cyr{CH}^1(K,S) \hookrightarrow \Hom(H^1(K,\Pic(\ovV)),\Q/\Z).$$

Pour une r\'ef\'erence r\'ecente pour ces faits sur les $K$-tores, voir \cite[Prop.
2.34]{bourqui}.

Le  th\'eor\`eme \ref{surfratavecpoint} donne donc naissance au diagramme commutatif de suites exactes :
\[\begin{CD}
 0 @>>>  A_{0}(V) @>>> H^1(K,S) @>>> {\mathcal N}(V) @>>> 0 \\
   &&        @VVV @VVV @VVV \\
  0 @>>> \displaystyle \bigoplus_{v}A_{0}(V_{v}) @>>> \displaystyle \bigoplus_{v}H^1(K_{v},S)
@>>> \displaystyle\bigoplus_{v}{\mathcal N}(V_{v}) @>>> 0.
  \end{CD}\]
  
La suite du serpent donne alors une suite exacte
$$0 \to \cyr{X}A_{0}(V) \to  \cyr{X}^1(K,S) \to  \cyr{X}{\mathcal N}(V) \to \cyr{CH}A_{0}(V) 
\to \cyr{CH}^1(k,S).$$
 On a donc :

\begin{prop}\label{equisurfrat}
Dans la situation ci-dessus, on a   \'equivalence entre les deux \'enonc\'es suivants
\begin{itemize}
\item[(a)] $  \cyr{X}{\mathcal N}(V)=0$ (Conjecture B de \cite{cts81}).
\item[(b)] L'injection  $ \cyr{X}A_{0}(V) \hookrightarrow  \cyr{X}^1(K,S)$ est un isomorphisme
et l'appli\-cation $\cyr{CH}A_{0}(V) \to \cyr{CH}^1(K,S)$ est injective.
\end{itemize}
\qed
\end{prop}

L'\'enonc\'e (b)  se r\'einterpr\`ete comme l'existence d'une suite exacte de grou\-pes finis
$$ 0 \to \cyr{X}^1(K,S) \to A_{0}(X) \to \bigoplus_{v} A_{0}(X_{v}) \to \Hom(H^1(K,\Pic(\ovV)),\Q/\Z)$$
soit encore, compte tenu de l'isomorphisme $\Br(V)/\Br(K) \iso H^1(K,\Pic(\ovV))$,
$$ 0 \to \cyr{X}^1(K,S) \to A_{0}(X) \to \bigoplus_{v} A_{0}(X_{v})  \to \Hom(\Br(V)/\Br(K),\Q/\Z).$$

Lorsque $K$ est un corps de nombres, l'existence d'une telle suite exacte est la conjecture A de
\cite{cts81}. Qu'elle r\'esulte de la conjecture B est d\'ej\`a mentionn\'e dans \cite{cts81}.

Pour une surface $V$  fibr\'ee en coniques sur $\P^1_{K}$, lorsque $K$ est un corps de nombres,
cette conjecture A fut \'etablie par Salberger \cite{salberger1}.
Nous sommes maintenant en mesure de donner un \'enonc\'e analogue sur un corps global $K$
de caract\'eristique positive. On observera que la m\'ethode de d\'emonstration est fort  diff\'erente
de celle de Salberger sur un corps de nombres.

\begin{theo}
Soit $K=\F(C)$ le corps des fonctions d'une courbe projective, lisse, g\'eom\'etriquement connexe
sur un corps fini $\F$ de caract\'eristique diff\'erente de $2$. Soit $X$ une $\F$-vari\'et\'e
projective, lisse, g\'eom\'etriquement connexe, de dimension $3$, \'equip\'ee d'un morphisme $X
\to \P^1_{C}$ de fibre g\'en\'erique une conique. Soit $V/K$ la fibre g\'en\'erique de $X \to
C$, qui est une surface fibr\'ee en coniques au-dessus de $\P^1_{K}$, et donc
g\'eom\'etriquement rationnelle. Supposons $I(V)=1$. On a alors une suite exacte
de groupes de torsion 2-primaire :
$$ 0 \to \cyr{X}^1(K,S) \to A_{0}(V) \to \bigoplus_{v} A_{0}(V_{v})  \to \Hom(\Br(V)/\Br(K),\Q/\Z).$$
\end{theo}

\begin{proof} 
Pour $V/K$ comme ci-dessus, il est connu que les groupes
intervenant dans cette suite sont de torsion 2-primaire, et m\^eme de $2$-torsion  (voir \cite{cts81}).
D'apr\`es la proposition \ref{troisdef}, on a $$\cyr{X}H^3_{\nr}(X,\Q_{2}/\Z_{2}(2)) \iso \cyr{X}{\mathcal N}(V)\{2\}.$$
D'apr\`es la proposition \ref{passageH3gen}
on a $$\cyr{X}H^3_{\nr}(V,\Q_{2}/\Z_{2}(2)) \hookrightarrow  H^3_{\nr}({X},\Q_{2}/\Z_{2}(2)).$$
Le th\'eor\`eme \ref{H3parimalasuresh} (Parimala et Suresh) donne $ H^3_{\nr}({X},\Q_{2}/\Z_{2}(2))=0$.
On a donc $ \cyr{X}{\mathcal N}(V)\{2\}=0$.  La d\'emonstration de la proposition \ref{equisurfrat}
valant composante $l$-primaire par composante $l$-primaire, ceci permet de conclure.
\end{proof}

\end{document}